\documentclass{amsart}
\usepackage{amssymb,amsmath,amsfonts,pictex,graphicx,etex,MnSymbol,color,tikz,nicefrac}

\newcommand{\Amax}{A_{{\sf max}}}
\newcommand{\A}{{\sf a}}
\newcommand{\B}{{\sf b}}

\newcommand{\E}{{\sf e}}

\newcommand{\V}{{\sf v}}
\newcommand{\U}{{\sf u}}
\newcommand{\W}{{\sf w}}
\newcommand{\X}{{\sf x}}
\newcommand{\Y}{{\sf y}}

\newcommand{\M}{\mathfrak{m}}
\newcommand{\uglyH}{\mathfrak{h}}
\newcommand{\Z}{\mathbb{Z}}
\newcommand{\N}{\mathbb{N}}
\newcommand{\R}{\mathbb{R}}
\newcommand{\Q}{\mathbb{Q}}

\newcommand{\trace}{{\sf trace}}

\newcommand{\HausD}{d_{{\sf Haus}}}
\newcommand{\zero}{{\sf 0}}
\newcommand{\z}{{\mathbf z}}

\newcommand{\I}{\mathbf I}
\newcommand{\std}{{\sf std}}
\newcommand{\CC}{{\sc cc}~}
\newcommand{\dCC}{d_{\hbox{\tiny \CC}}\!}

\newcommand{\sdot}{\! \cdot \!}

\newcommand{\eS}{\mathcal{S}}
\newcommand{\eN}{\mathcal{N}}

\newcommand{\NN}{{\sf N}}
\newcommand{\G}{\mathcal{G}}

\DeclareMathOperator{\vol}{vol}

\DeclareMathOperator{\diam}{diam}
\DeclareMathOperator{\panel}{{\sf Panel}}
\DeclareMathOperator{\Sreg}{\eS_{\sf reg}}
\DeclareMathOperator{\Suns}{\eS_{\sf uns}}
\newcommand{\Sreghat}{{\widehat{\eS}_{\sf reg}}}
\newcommand{\Sunshat}{{\widehat{\eS}_{\sf uns}}}

\newcommand{\Foot}{{\sf Foot}}
\newcommand{\VS}{V}
\newcommand{\Vreg}{V_{\hbox{\small \sf reg}}}
\newcommand{\Vuns}{V_{\hbox{\small \sf uns}}}
\DeclareMathOperator{\Prob}{Prob}

\newcommand{\muL}{\mu\raisebox{-3pt}{\scriptsize{\ensuremath L}}}
\newcommand{\normL}[1]{\| #1 \| \raisebox{-3pt}{\scriptsize{\ensuremath L}}}

\newtheorem{Thm}{Theorem}

\newtheorem{prop}[Thm]{Proposition}

\newtheorem{lemma}[Thm]{Lemma}

\newtheorem{defn}[Thm]{Definition}

\newtheorem*{struc-thm}{Structure Theorem}
\newtheorem*{instab-thm}{Instability Theorem}
\newtheorem*{counting-thm}{Counting Theorem}
\newtheorem*{track-lem}{Tracking Lemma}

\theoremstyle{definition}
\newtheorem{remark}[Thm]{Remark}
\newtheorem{example}[Thm]{Example}

\begin{document}

\title{Fine asymptotic geometry in the Heisenberg group}
\author[Duchin]{Moon Duchin}
\address{Department of Mathematics\\  
530 Church Street\\
Ann Arbor, MI 48109-1043
}
\email{mduchin@umich.edu}
\urladdr{http://www.math.lsa.umich.edu/$\sim$mduchin/}

\author[Mooney]{Christopher Mooney}
\address{Department of Mathematics\\  
530 Church Street\\
Ann Arbor, MI 48109-1043
}
\email{cpmooney@umich.edu}
\urladdr{http://www.math.lsa.umich.edu/$\sim$cpmooney/}

\thanks{The first author was partially supported by NSF DMS-0906086. 
The second author was partially supported by NSF RTG-0602191.}

\date{\today}
\begin{abstract} 
For every finite generating set on the integer Heisenberg group $H(\Z)$, 
Pansu showed that 
the word  metric has the large-scale structure of a Carnot-Carath\'eodory Finsler metric on the real 
Heisenberg group $H(\R)$.  
We study the properties of those limit metrics and obtain results about the geometry of 
word metrics that reflect the dependence on generators.

For example we will study the probability that a group element has all of its geodesic spellings
sublinearly close together, relative to word length.  In free abelian groups of rank 
at least two, that probability is 0; in unbounded hyperbolic groups, the probability is 1.  
In $H(\Z)$ it is a rational number strictly between 0 and 1 that depends on the generating set; with respect
to the standard generators, the probability is precisely $19/31$.
\end{abstract}
\maketitle

\section{Introduction}

In this paper we will focus on the 3-dimensional integer Heisenberg group $H(\Z)$, a uniform lattice
in the real Heisenberg group $H(\R)$, with the word metric coming from a finite generating set.  
We will study large-scale
geometric properties that are  sensitive to the choice of generators, such as the shape 
and stability of long geodesics, and asymptotic density of  classes of elements.  

We will take an approach based on work of Pansu and Breuillard \cite{pansu,breuillard} exploiting the fact 
 that all word metrics on $H(\Z)$ 
 have the large-scale structure of Carnot-Carath\'e\-odory Finsler metrics
on $H(\R)$; that is, 
the Lie group with its \CC metric is the asymptotic cone of the lattice with
the chosen word metric.  
However, we will study large-scale geometric properties  that are lost when one 
passes to the asymptotic cone.  (Note that studying the Heisenberg group 
up to quasi-isometry is even coarser, identifying all
word metrics as well as all left-invariant metrics on $H(\R)$ with one another.)
Thus we develop technology for working with these word metrics that uses
the homogeneous dilation at finite scale together with combinatorial arguments, 
instead of the usual tools of coarse geometry.  We think of this
set of questions and techniques as belonging to  ``fine asymptotic geometry."

The geometry of nilpotent groups is closely tied to geometric minimax problems,
so Section~\ref{isoper} is devoted to the solution of relevant isoperimetric problems in 
normed planes.
In Section~\ref{structure}, we apply those findings,
giving a very explicit description of the limit metric associated to an arbitrary word metric.
The main result of this section 
is Theorem~\ref{limitshapebypanels}, which is summarized here:

\begin{struc-thm}
For any finite symmetric generating set $S$ of $H(\Z)$, the limiting \CC metric admits a complete
description of its geodesics, classified into 
no more than $|S|^2-2|S|$ combinatorial types of {\em regular geodesics} and no more than 
$|S|$ types of {\em unstable geodesics}.  The unit sphere $\eS$ in the \CC metric is a piecewise union
of the graphs of quadratic polynomials over finitely many quadrilateral regions in the $xy$-plane.
\end{struc-thm}

As a corollary of this description, we can analyze the uniqueness of geodesics in these 
limit metrics, describing exactly which points $\X\in H(\R)$ are reached by more than one 
geodesic segment based at $\zero$.

We use this structure theorem 
in Section~\ref{instability} to understand the geodesics in Cayley graphs for $H(\Z)$, 
showing that word geodesics may be very unruly, but they are 
 tracked to within a controlled distance  by much better-behaved \CC geodesics:
the Hausdorff distance is of lower-order growth than the length of the word.  (See
Lemma~\ref{tracking}.)

\begin{track-lem}
Away from a certain unstable locus of endpoints,  word geodesics
are sublinearly tracked by \CC geodesics.
\end{track-lem}

This opens the door to the study of geometric probability in the discrete Heisenberg groups,
and we give an illustrative application in Theorem~\ref{instab},
 quantifying the instability of word geodesics.  
Elements of a group may be said to be geodesically stable
with respect to a generating set if, in the Cayley graph, the Hausdorff
distance between the geodesic spellings of the
word can only differ from each other sublinearly in the length of
the word.  When this is made precise, one can see that stable
elements have zero density in 
 free abelian groups, but full density in infinite hyperbolic groups.  The situation is different
 in the Heisenberg group.

\begin{instab-thm}
The asymptotic density of geodesically stable elements of $H(\Z)$ is a rational number
strictly between $0$ and $1$, depending on the generating set.
\end{instab-thm}

Then in Section~\ref{counting}, we apply the  structure theorem to 
a generalized Gauss circle problem:  
 Theorem~\ref{countingthm} asserts that in every polygonal \CC metric---a 
class that includes all limits of word metrics---we can count lattice
points to first order in the annular shell between balls.  

\begin{counting-thm}
For any polygonal \CC metric, the number of lattice points between the spheres of radius
$n$ and $n-1$ is equal to $4Vn^3+O(n^2)$, where $V$ is the volume of the \CC unit ball.
\end{counting-thm}

Note that this is strictly better than what is logically implied by the best known estimate
of the number of lattice points in the ball of radius $n$, which is $Vn^4+O(n^3)$ by \cite{stoll2}.
(In fact, the full theorem is more precise than what is stated above:  we can
count points in the annulus in any radial direction.)

Finally, we show in Section~\ref{std-sec} that  for $S=\std$, the cardinality of the 
discrete sphere $S_n$ agrees with this annular lattice point count  to first order.
The counting measure on discrete spheres is thus shown to limit to a cone measure on 
the limit shape $\eS$, for which we have a concrete and finite description.    
This enables spherical averaging---sharper than asymptotic density calculations,
which are averages over balls---with respect
to the standard generators.


\subsection{Nilpotent groups and the Heisenberg group}
We review a few facts about the sub-Riemannian geometry of the Heisenberg group;
see ~\cite{capogna} for a comprehensive reference.

First, we consider the 3-dimensional integer Heisenberg group $H(\Z)$, the
subgroup of matrices with integer entries  in  $H(\R)$ (the real Heisenberg group).  
Let $\M$ denote the horizontal subspace of the Lie algebra
$\uglyH$ of $H(\R)$; that is, the span of the tangent vectors 
$X_0=\left( \begin{smallmatrix}
	0&1 &0  \\
	0 &0&0\\
	0&0&0
 \end{smallmatrix} \right)$
and 
$Y_0=\left( \begin{smallmatrix}
	0&0&0  \\
	0 &0&1\\
	0&0&0
 \end{smallmatrix} \right)$.  There are {\em horizontal planes} at every point in $H(\R)$, 
 generated by pushing 
around $X_0$ and $Y_0$ by left multiplication in $H(\R)$ to produce left-invariant vector fields $X$ and $Y$.  
We say that a curve  in $H(\R)$ is {\em admissible} if all of its tangent vectors lie in these
horizontal planes;
it is easily verified that the curve $\gamma=(\gamma_1,\gamma_2,\gamma_3)$ is admissible if and only if
$$\gamma_3' = \frac 12 (\gamma_1\gamma_2' - \gamma_2\gamma_1').$$

We will use the {\em exponential coordinates} on $H$ given by the following representation:
$$(x,y,z) \leftrightarrow \begin{pmatrix}
	1&x&z + \frac 12 xy  \\
	0 &1&y\\
	0&0&1
 \end{pmatrix} .$$
These coordinates have the property that $(x,y,z)^n=(nx,ny,nz)$.

For integers $x$ and $y$, define $\epsilon(x,y)$ to be $1/2$ if $x$ and $y$ are both odd, and $0$ otherwise.
In these coordinates, $H(\Z)$ looks just like the standard lattice $\Z^3\subset \R^3$ shifted by $\epsilon$ in the $z$ direction.

We can regard $\M$ as a copy of $\R^2$ and 
make use of the  linear projection $\pi: H(\R) \to \M$ given by $(x,y,z)\mapsto (x,y)$.

\begin{lemma}[Standard fact from Heisenberg geometry]
\label{balayage}
For any path $\gamma=(\gamma_1,\gamma_2)$ from 
$(0,0)$ to $(x,y)$, there is a unique admissible curve $\overline\gamma=(\gamma_1,\gamma_2,\gamma_3)$. 
The lifted curve connects the origin to the point $(x,y,z)$, where
$z$  is equal to the signed Euclidean area of $R$, the region in the plane
enclosed by $\gamma$ and a straight chord from
$(0,0)$ to $(x,y)$.
\end{lemma}

\begin{proof}
To construct the lift, we just need to supply the coordinate function $\gamma_3$, which can be done
by integrating in the defining formula for admissibility.
Note that $x(t)y'(t)-y(t)x'(t)$ 
is identically zero along a linear parametrization of the straight chord
between the endpoints, so we can concatenate that chord with the path in the plane to obtain the closed
curve $\partial R$.  We get
$$z= \int_{\partial R} \gamma_1\gamma_2' - \gamma_2\gamma_1' = \int_R dx\wedge dy,$$
which is the area of $R$.
\end{proof}

From now on the area of this region $R$ will be called the {\em balayage area} associated
to a curve $\gamma$.

Choose a centrally symmetric convex body  $Q\subset\M$ with boundary $L$
and let $\normL\cdot$ denote the corresponding  norm; this is the norm having $Q$ as its unit ball
and $L$ as its unit sphere.  We will use the notation
$\M_L=(\M,\normL\cdot)$ for this normed plane.
This induces a Carnot-Carath\'eodory Finsler metric (from now on,  \CC metric)
on the Lie group:  the distance between two points
is the infimal length of an admissible path between them, measured by using the norm on the tangent 
vectors.
(This is well defined for admissible curves because the tangent vectors lie in copies of
$\M$ pushed around by multiplication in $H(\R)$, and it is a classical fact that this gives $H(\R)$ the structure
of a geodesic space.)
Note that measuring the length of an admissible curve in $H(\R)$ with respect to this \CC metric
is equivalent to measuring the length of its projection to $\M$ with respect to the norm
$\normL{\cdot}$.  If $L$ is a polygon, we will call the induced metric a {\em polygonal \CC metric}
on $H(\R)$.  (In arbitrary dimension, if $L$  is a polytope, the associated norms are sometimes
called {\em crystalline}.)

These \CC metrics are equipped with a dilation $\delta_t(x,y,z)=(tx,ty,t^2z)$ that scales
lengths and distances by $t$, and areas in $\M$ by $t^2$.  For any set $E\subset H(R)$, let
$\Delta E$ denote its full cone under dilation, and let $\hat E$ denote its dilation-cone to the origin:
$$\Delta E = \{ \delta_t(\X) : \X\in E, t\ge 0\} \quad ; 
\qquad \hat E = \{ \delta_t(\X) : \X\in E, 0\le t\le 1\} .$$ 

As a consequence of the connection between height and balayage area, we have a criterion for geodesity
in the \CC metric:
a curve $\gamma$ in $\M$ based at $(0,0)$ lifts to a geodesic in $H(\R)$
iff its $L$-length is minimal among all curves with the same endpoints and enclosing the same area.

We will consider geometric probability, or  {\em asymptotic density},
in the Heisenberg groups; for
$U\subset H(\Z)$ and $V\subset H(\R)$, let $B_t=B_t(\zero)$ be the ball of radius $t$ about the 
origin.  Then
\[
	\Prob(U):= \lim_{n\to\infty}
		\frac{|B_n \cap U|}{|B_n|} \quad ; \qquad 
	\Prob(V):=\lim_{r\to\infty}
		\frac{\vol(B_r\cap V)}{\vol(B_r)}.
\]
Note that Lebesgue measure has the right invariance properties
to coincide with Haar measure on $H(\R)$  up to a scalar, so we
will be able to talk about ratios of volumes unambiguously.  

\subsection{Limit shapes, limit metrics,  and limit measure}

We have the following extremely general theorem, first proven for nilpotent groups by Pansu,
and extended to all periodic pseudometrics on 
simply connected solvable Lie groups of polynomial growth by Breuillard
\cite{breuillard}.
Here we just state it for the very special case of word metrics on $H(\Z)\le H(\R)$.  We will always assume
that generating sets are symmetric ($S=-S$).

\begin{Thm}[Pansu \cite{pansu}]
\label{pansu-thm}
Consider  a word metric on $H(\Z)$ given by a finite generating set $S$.
Let $\pi(S)$ be the linear projection of the generators $S$ to the horizontal subspace $\M$,  let 
$Q$ be the convex hull of $\pi(S)$ and $L$ its boundary polygon.
Then the limit $$\eS = \lim_{n\to\infty}  \delta_{\frac 1n}S_n \ ,$$
exists (as a Gromov-Hausdorff limit, say) and is equal to the unit sphere in the \CC metric induced by 
the norm $\normL{\cdot}$ on $\M$.

Equivalently, $$\lim_{\X\to\infty} \frac{|\X|_S}{\dCC(\X,\zero)}=1.$$
\end{Thm}

\begin{Thm}[Krat \cite{krat}]\label{bdd-diff}
Furthermore, there exists a constant $K=K(S)$ such that
$$\dCC(\X,\zero) -K \  \le \   |\X|_S \  \le \  \dCC(\X,\zero) +K$$
for all $\X\in H(\Z)$.
\end{Thm}

The second result, bounded difference between the word metric and the \CC metric, is stronger here
in $H(\Z)$ than for the general case, where one only has that the ratio goes to $1$.  
In fact Breuillard showed in \cite{breuillard} that there exist 2-step nilpotent groups where bounded
difference fails.  

For the word metric $(H(\Z),S)$, we will call $\eS$ and the induced
\CC metric the {\em limit shape} and {\em limit metric}, respectively. 
The volume of the unit ball in the \CC metric will turn out to be
an important number to attach to both the word metric and the \CC
metric itself; we will denote it by
$$\VS=\VS(S)=\VS(L):=\vol(\hat\eS).$$
The Hausdorff dimension of $H(\R)$ 
is equal to $4$, despite the topological dimension of $3$, as one can see 
clearly by considering the growth of a cube under the dilation.
Thus the volume of the ball of radius $r$ in the \CC metric equals
$\VS\sdot r^4$, and in $H(\Z)$, we have $|B_n|=\VS\sdot n^4 + O(n^3)$.


In the previous paper  \cite{DLM}, the case of  $\Z^d$ was studied
in detail, and there we also obtained a {\em limit measure} for a
generating set $S$.  In that case $Q$ is the convex hull in $\R^d$
of the generators $S$ themselves, and $L$ is again its boundary.
The limit shape here is given by
$L=\lim\limits_{n\to\infty} \frac 1n S_n$.
In \cite{DLM} it was shown that counting measure on $\frac 1n S_n$
converges to a measure on $L$  called  {\em cone measure}, which
assigns to a measurable set $\sigma\subset L$ the measure
$\muL(\sigma)=\frac{\vol(\hat\sigma)}{\vol(Q)}$.  
(That is, it is the proportion of the area of $Q$ that is subtended
by $\sigma$, in this case by Euclidean dilation.)
This allows us to reduce the problem of averaging over larger and larger 
$S_n$ with respect to counting measure
to the finite problem of averaging over $L$ against cone measure.

Below, in Theorem~\ref{std}, we will obtain the analogous result for $H(\Z)$ with
respect to the standard generators, which enables us to calculate
spherical averages; conjecturally, cone measure is the limit measure
for all finite generating sets.  


\subsection{Relationship  to prior literature on $H(\Z)$}

Pansu's seminal paper from 1983 \cite{pansu}, based on his dissertation,
shows that the ratio of limit metric to word metric goes to one
 for all virtually nilpotent groups with finite generating sets, as discussed above, and concludes that
these word metrics on $H(\Z)$ satisfy 
$|B_n|=\alpha\cdot n^d +o(n^d)$.   Breuillard's preprint from 2007 \cite{breuillard}
extends this result (that the growth 
has a well-defined leading coefficient) to a much larger class of groups,
and computes the coefficient.    Breuillard also gives extremely concrete 
geometric constructions in that paper, especially in the Appendix, where  some explicit
computations are shown for $(H(\Z),\std)$ that inspired the current work.

M.\! Shapiro wrote a study of  $H(\Z)$ with respect to its standard generators in 1989 \cite{shapiro}, 
giving a description of $S_n$, then
computing an exact formula for $|S_n|$ as a polynomial function of $n$.
This remarkable formula has not only a well-defined leading coefficient, but in fact
all of the coefficients are computed and only the constant term oscillates (in fact, between 
twelve different values).  
He uses this study to establish that $H(\Z)$ has infinitely many cone types, and 
is almost convex in the sense of Cannon. 

Stoll 1998 \cite{stoll2} refines Pansu's asymptotics by showing that
for 2-step nilpotent  groups with finite generating sets, 
$|B_n|=\alpha\sdot n^d + O(n^{d-1})$, and this error term is ``often" sharp---compare
our counting theorem (Thm~\ref{countingthm}).
In a separate paper, Stoll had showed that 
2-step nilpotent groups with infinite cyclic derived subgroup have a finite generating set
for which the growth series is rational; on the other hand, higher Heisenberg groups (of dimension 5
and up) have a finite generating set for which the growth series is transcendental  \cite{stoll1}.
These results are recovered by Breuillard by geometric methods.

Blach\`ere 2003  \cite{blachere} computes an exact 
word length formula with respect to 
the standard generators, using it to show ``almost-connectedness" of discrete
spheres $S_n$. 
We note that the calculations of Shapiro and Blach\`ere for $S_n$ with respect to $S=\std$ 
are somewhat combinatorial, 
and that they are essentially equivalent to the more geometric description of $S_n$
given below in \S\ref{std-sec} (which also give directional information).

Finally, in work based on her dissertation, 
Dani 2007 \cite{dani} has some striking results on asymptotic density:
for any virtually nilpotent group, she gives an exact formula for the density of finite-order elements.
Further, she shows that the values of these densities range over $\Q\cap [0,1]$ as the groups vary.

\subsection*{Acknowledgments}

Many thanks to Ralf Spatzier, Alex Eskin, and Emmanuel Breuillard for useful conversations
on numerous aspects of the problems considered here.

\section{Isoperimetry and isoarea in normed planes}\label{isoper}
\subsection{The classical problem}

As seen above,
understanding \CC geodesics in $H(\R)$ amounts to the problem
of finding paths in $\M$ of minimal $L$-length given their endpoints and 
balayage area.
For closed loops, this is just the classical isoperimetric problem
in the normed space $\M_L$, which was elegantly solved by Busemann in the 1940s as follows.

Let the {\em polar dual} $Q^*$ of a convex, centrally symmetric polygon $Q$ be defined 
by $$Q^*=\{\X\in\M : \ \X\cdot\Y\le 1 \quad  \forall \Y\in Q\}.$$
Then if $\U$ and $\V$ are successive vertices of $Q$,  
$Q^*$ has a vertex $\A$ corresponding to the edge between them
which is the solution to $\A\cdot\U = \A\cdot\V=1$.
(That is, if the side is called $\sigma$ and the line through the origin perpendicular to $\sigma$ is called
$\ell$, then the vertex dual to $\sigma$ is a vector pointing in the direction of $\ell$ and whose length
is the reciprocal of the length of the projection of $\U$ or $\V$ to $\ell$.)
In particular, if $Q$ is a polygon with $2N$ sides, then $Q^*$ is also a $2N$-gon.

\begin{Thm}[Busemann \cite{busemann}]
Consider the norm on $\R^2$ induced by a convex centrally symmetric body $Q$.  Then the maximal ratio of 
enclosed area to perimeter is uniquely achieved (up to scale) by the {\em isoperimetrix} $I$, 
which is defined to be the boundary of the rotate by $\pi/2$
of the polar dual $Q^*$ of $Q$.  
\end{Thm}

If $Q$ is a polygon, then the sides of $I$ are parallel to the
vertex directions of $Q$.  Here is the calculation:
suppose $a$ and $b$ are successive vertices of $Q^*$, so they are
dual to successive edges of $Q$, sharing a vertex $v$.  Then a side
of $I$ is obtained by rotating $a-b$ by $\pi/2$, so that side of
$I$ is perpendicular to $a-b$.  But since $a\cdot v=1$ and $b\cdot v=1$
(by construction of $a$ and $b$), we have $(a-b)\cdot v=0$, which
shows that the side of $I$ is parallel to $v$.

Most of the rest of this section is devoted to solving modifications
of this problem that are relevant for the nilpotent geometry.  We
will introduce classes of curves called {\em trace paths} and
{\em beelines}, then show that these provide solutions to the
Dido problem and the dual isoarea problem in $\M_L$.

\subsection{Trace paths}

Let $I_1$ be a scaled copy of $I$,  scaled to have perimeter one
in the $L$-norm.  Fix a parametrized $1$-periodic curve
$\iota: \R \to \R^2$ so that the image of $\iota$ is $I_1$,
the parametrization is by arclength relative to the norm,
and $\iota$ traverses $I_1$ counterclockwise.

For any choice of parameters $t\in [0,1]$ and $T\in (t,t+1]$, 
let $$\trace_{t,T}(s)=\frac{\iota\bigl(sT+t-st\bigr)-\iota(t)}{T-t}, \qquad s\in [0,1].$$
We call these {\em trace paths}:  they are parametrized paths of length one in the norm, whose
shape follows along
the perimeter of $\frac 1{T-t} I_1$ (a scaled copy of the isoperimetrix).  These trace paths
start at the origin and  therefore must terminate inside the unit ball in the norm (namely $Q$).
See Figure~\ref{combx} for a classification of trace paths for two different choices of $Q$.

\begin{defn} For a given normed plane $\M_L$,
let the {\em balayage function} $A:Q\to\R$ denote the maximal 
balayage area among all paths of length $1$ from $(0,0)$ to $(x,y)$.
\end{defn}

\begin{lemma}[Trace paths solve the Dido problem]
\label{dido}
Given $(x,y)\in Q$, 
a path $\gamma$ of $L$-length $1$ connecting $(0,0)$ and $(x,y)$ has maximal 
balayage area among all such paths if and only if it is a trace path:
 $\gamma=\trace_{t,T}$
where $\trace_{t,T}(1)=(x,y)$.  
\end{lemma}

\begin{proof}  
Consider a particular trace path with $\trace_{t,T}(1)=(x,y)$.
Let $\alpha$ be the straight chord from $(0,0)$ to $(x,y)$
and let $\lambda=\frac{1}{T-t}$.  Let $A_{\sf tr}$ denote the area enclosed by
$\trace_{t,T}$ and $\alpha$; we want to show that this is maximal.
Note that the trace path can be completed to $\lambda I_1$ (a scaled copy of $I$)
by extending the domain up to $s=\lambda$.  Thus $\alpha$ can be realized as a chord of $\lambda I_1$ 
that cuts it into two pieces, and $A_{\sf tr}$ is the area of one of those two pieces.
But now suppose there were a different path 
$\gamma$ in $\R^2$, of length one in the $L$-norm, connecting $(0,0)$ to $(x,y)$
while enclosing maximal area  $A(x,y)\ge A_{\sf tr}$ with the chord $\alpha$.
Maximality of area implies convexity,
so in particular the path $\gamma$ always stays on one side of the line through 
$\alpha$.  But then $\gamma$ can also be concatenated with $\trace_{t,T}([1, \lambda])$, producing
a figure which must be a copy of $\lambda I_1$, by uniqueness of the isoperimetrix.  Thus
$A_{\sf tr}=A(x,y)$, showing that trace paths enclose maximal area among all paths of length one, as
desired.  Indeed, this 
also means that the image of $\gamma$ is uniquely determined, which forces it to be a trace path,
though it is possible that
$\gamma=\trace_{t,T}=\trace_{t',T'}$ for another 
choice of parameters with $t'=t+k$, $T'=T+k$.
\end{proof}

From here on we will assume that $Q$ is a $2N$-gon, in which case
$I_1$ will also be a $2N$-gon and $\iota$ will be piecewise
affine: $\iota(s)=s\A+\B$ for
fixed vectors $\A,\B$ when $s$ is in the subinterval of values
corresponding to a side of $I$.
Let $\iota$ be chosen so that $\iota(0)$ is a vertex in a copy of $I_1$
and let $\sigma_1,\sigma_2, \ldots, \sigma_{2N}$ be the sides
listed counterclockwise.
Let $\ell_i$ be the sidelength of $\sigma_i$ for all $i$, so that $\sum \ell_i=1$.

Define $$Q_{ij}:=\left\{ \frac{\iota(T)-\iota(t)}{T-t}  :   \quad \iota(t)\in \sigma_i, \ \  \iota(T)\in \sigma_j\right\}.$$
This is the subset of $Q$ consisting 
of endpoints $\trace_{t,T}(1)$ of those trace paths whose shape on $I_1$ begins on the $i$th side
and ends on the $j$th side.

We will say that two continuous, piecewise linear paths have the {\em same shape}
 if there exist vector directions 
$\V_1,\ldots,\V_k$ such that each path follows $a_1\V_1,a_2\V_2,\ldots,a_k\V_k$
in the same order, with $a_1,a_k\ge 0$, $a_2,\ldots,a_{k-1}>0$, and $\V_i\neq \V_{i+1}$ for all $i$.

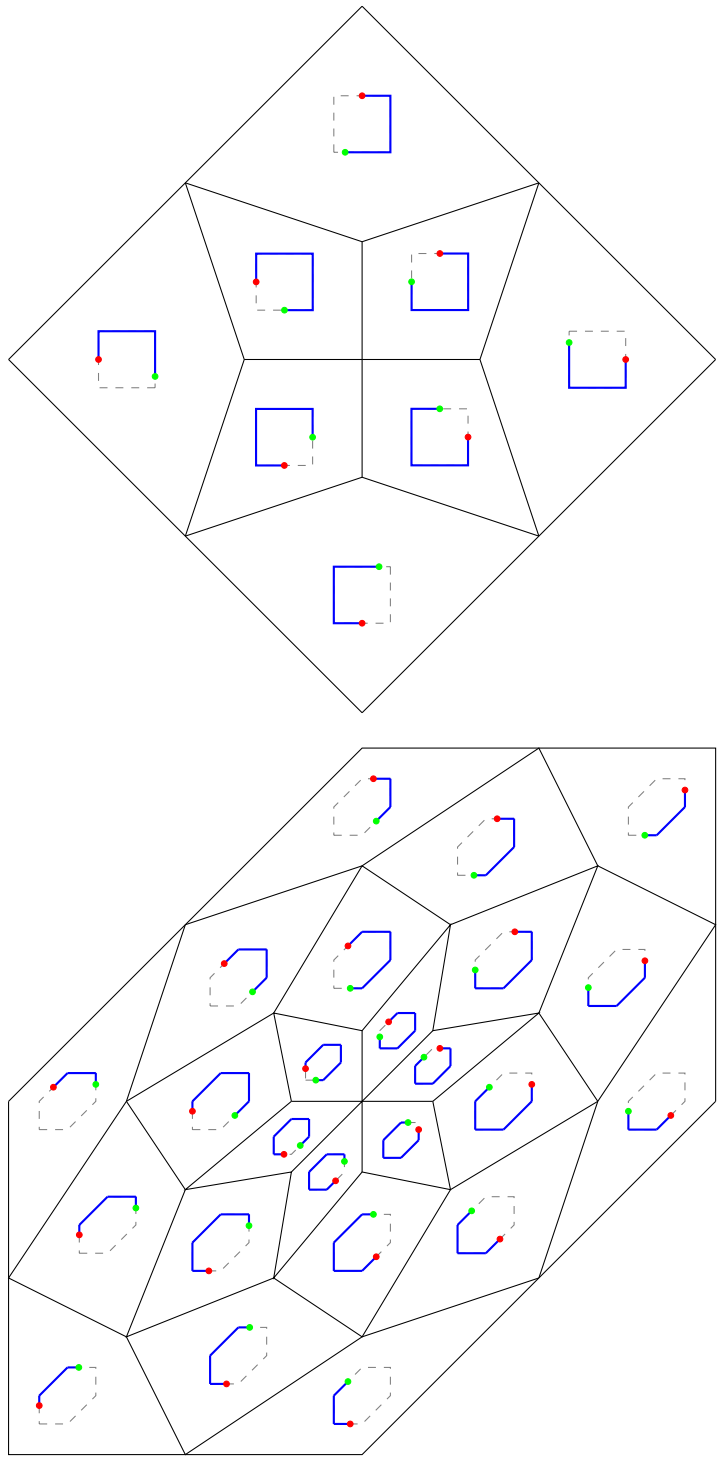
\begin{figure}
\begin{tikzpicture}[scale=4.7]
\foreach \ang in {0,90,180,270}
{\def\r{.08}
\def\isoper{+(-\r,-\r) rectangle +(\r,\r)}
\def\LRtrace{+(-.6*\r,-\r)--+(\r,-\r)--+(\r,\r)--+(0,\r)}
\def\LTtrace{+(-\r,0)--+(-\r,-\r)--+(\r,-\r)--+(\r,\r)--+(0,\r)}
{\begin{scope}[rotate=\ang]
\coordinate (x) at (0,2/3);
\coordinate (y) at (.22,.22);
\draw (1,0)--(0,1);
\draw [gray,dashed] (x)\isoper;
\draw [thick, blue] (x)\LRtrace;
\filldraw [green] (x)+(-.6*\r,-\r) circle (.008);
\filldraw [red] (x)+(0,\r) circle (.008);
\draw [gray,dashed] (y)\isoper;
\draw [thick, blue] (y)\LTtrace;
\filldraw [green] (y)+(-\r,0) circle (.008);
\filldraw [red] (y)+(0,\r) circle (.008);
\draw (1/2,1/2)--(0,1/3)--(-1/2,1/2);
\draw (0,0)--(1/3,0);
\end{scope}}}

\begin{scope}[yshift=-2.1cm]
\def\B{.08}
\def\b{.05}
\def\isoper{+(-\r,-\r)--+(0,-\r)--+(\r,0)--+(\r,\r)--+(0,\r)--+(-\r,0)--cycle}
\def\top{+(.4*\r,\r)}  \def\nw{+(-.5*\r,.5*\r)}  \def\left{+(-\r,-.35*\r)} \def\bot{+(-.42*\r,-\r)}  \def\se{+(.5*\r,-.5*\r)}   \def\right{+(\r,.6*\r)}
\def\vtop{+(0,\r)}  \def\vnw{+(-\r,0)}  \def\vleft{+(-\r,-\r)} \def\vbot{+(0,-\r)}  \def\vse{+(\r,0)}   \def\vright{+(\r,\r)}
\def\topstart{\top--\vtop} \def\nwstart{\nw--\vnw} \def\leftstart{\left--\vleft} \def\botstart{\bot--\vbot} \def\sestart{\se--\vse} \def\rightstart{\right--\vright}
\def\topstop{\vright--\top} \def\nwstop{\vtop--\nw} \def\leftstop{\vnw--\left} \def\botstop{\vleft--\bot} \def\sestop{\vbot--\se} \def\rightstop{\vse--\right}
\def\nwside{\vtop--\vnw}  \def\leftside{\vnw--\vleft} \def\botside{\vleft--\vbot}  \def\seside{\vbot--\vse} \def\rightside{\vse--\vright} \def\topside{\vright--\vtop}

\coordinate (a) at (5/6,5/6); \coordinate (A) at (-5/6,-5/6);
\coordinate (b) at (5/6,0); \coordinate (B) at (-5/6,0);
\coordinate (c) at (.72,.35); \coordinate (C) at (-.72,-.35);
\coordinate (d) at (.35,.72); \coordinate (D) at (-.35,-.72);
\coordinate (e) at (.4,.4); \coordinate (E) at (-.4,-.4);
\coordinate (f) at (0,5/6); \coordinate (F) at (0,-5/6); 
\coordinate (g) at (0,.4); \coordinate (G) at (0,-.4);
\coordinate (h) at (-.35,.35); \coordinate (H) at (.35,-.35);
\coordinate (i) at (.4,0); \coordinate (I) at (-.4,0);
\coordinate (j) at (-.11,.11); \coordinate (J) at (.11,-.11);
\coordinate (k) at (.2,.1); \coordinate (K) at (-.2,-.1);
\coordinate (l) at (.1,.2); \coordinate (L) at (-.1,-.2);
\foreach \r/\x in {\B/a,\B/b,\B/c,\B/d,\B/e,\B/f,\B/g,\B/h,\B/i,\b/j,\b/k,\b/l,\B/A,\B/B,\B/C,\B/D,\B/E,\B/F,\B/G,\B/H,\B/I,\b/J,\b/K,\b/L}
{\draw [gray,dashed] (\x)\isoper;}
\foreach \r/\x in {\B/A,\B/D,\B/G,\b/J}{\draw [thick,blue] (\x)\topstart; \filldraw [green] (\x)\top circle (0.008);}
\foreach \r/\x in {\B/F,\B/H,\B/i,\b/k}{\draw [thick,blue] (\x)\nwstart; \filldraw [green] (\x)\nw circle (0.008);}
\foreach \r/\x in {\B/b,\B/c,\B/e,\b/l}{\draw [thick,blue] (\x)\leftstart; \filldraw [green] (\x)\left circle (0.008);}
\foreach \r/\x in {\B/a,\B/d,\B/g,\b/j}{\draw [thick,blue] (\x)\botstart; \filldraw [green] (\x)\bot circle (0.008);}
\foreach \r/\x in {\B/f,\B/h,\B/I,\b/K}{\draw [thick,blue] (\x)\sestart; \filldraw [green] (\x)\se circle (0.008);}
\foreach \r/\x in {\B/B,\B/C,\B/E,\b/L}{\draw [thick,blue] (\x)\rightstart; \filldraw [green] (\x)\right circle (0.008);}
\foreach \r/\x in {\B/f,\B/d,\B/e,\b/k}{\draw [thick,blue] (\x)\topstop; \filldraw [red] (\x)\top circle (0.008);}
\foreach \r/\x in {\B/B,\B/h,\B/g,\b/l}{\draw [thick,blue] (\x)\nwstop; \filldraw [red] (\x)\nw circle (0.008);}
\foreach \r/\x in {\B/A,\B/C,\B/I,\b/j}{\draw [thick,blue] (\x)\leftstop; \filldraw [red] (\x)\left circle (0.008);}
\foreach \r/\x in {\B/F,\B/D,\B/E,\b/K}{\draw [thick,blue] (\x)\botstop; \filldraw [red] (\x)\bot circle (0.008);}
\foreach \r/\x in {\B/b,\B/H,\B/G,\b/L}{\draw [thick,blue] (\x)\sestop; \filldraw [red] (\x)\se circle (0.008);}
\foreach \r/\x in {\B/a,\B/c,\B/i,\b/J}{\draw [thick,blue] (\x)\rightstop; \filldraw [red] (\x)\right circle (0.008);}
\foreach \r/\x in {\B/B,\B/h,\B/g,\b/j,\B/C,\B/E,\b/l,\b/K,\b/L,\B/I} {\draw [thick,blue] (\x)\topside;}
\foreach \r/\x in {\B/A,\B/E,\B/C,\B/I,\b/j,\B/D,\B/G,\b/J,\b/K,\b/L} {\draw [thick,blue] (\x)\nwside;}
\foreach \r/\x in {\B/F,\B/H,\B/i,\b/k,\B/D,\B/G,\b/J,\b/K,\b/L,\B/E} {\draw [thick,blue] (\x)\leftside;}
\foreach \r/\x in {\B/b,\B/H,\B/G,\b/J,\B/c,\B/e,\b/L,\b/k,\b/l,\B/i} {\draw [thick,blue] (\x)\botside;}
\foreach \r/\x in {\B/a,\B/e,\B/c,\B/i,\b/J,\B/d,\B/g,\b/j,\b/k,\b/l} {\draw [thick,blue] (\x)\seside;}
\foreach \r/\x in {\B/f,\B/h,\B/I,\b/K,\B/d,\B/g,\b/j,\b/k,\b/l,\B/e} {\draw [thick,blue] (\x)\rightside;}

\foreach \ang in {0,180}
{\begin{scope}[rotate=\ang]
\draw (1,0)--(1,1)--(0,1)--(-1,0);
\draw (-1/2,1/2)--(0,2/3)--(1/2,1)--(2/3,2/3)--(1,1/2)--(2/3,0)--(1/2,-1/2);
\draw (-2/3,0)--(-1/4,1/4)--(0,2/3)--(1/4,1/2)--(2/3,2/3)--(1/2,1/4)--(2/3,0);
\draw (-1/5,0)--(-1/4,1/4)--(0,1/5)--(1/4,1/2)--(1/5,1/5)--(1/2,1/4)--(1/5,0);
\draw (0,1/5)--(0,0)--(1/5,1/5);
\draw (0,0)--(1/5,0);
\end{scope}}
\end{scope}
\end{tikzpicture}

\caption{The combinatorics of trace paths for a square and a hexagon.  
We can count $8$ and $24$  nondegenerate quadrilaterals, respectively, with 
constant combinatorics on each.\label{combx}}
\end{figure}

\begin{Thm}[Combinatorics of trace paths]\label{tracepaths}
If  $Q$ is a polygon with $2N$ sides, then
there are $4N^2-4N$ quadrilaterals $Q_{ij}$, covering $Q$ and overlapping
only on their boundaries, 
such that the trace paths
have ``constant combinatorics" on each piece:  that is,  the trace paths 
ending at  any two points 
in $Q_{ij}$ have the same shape.
\end{Thm}

\begin{proof}
We will show that the quadrilaterals $Q_{ii}$ and $Q_{i,i+1}$ are degenerate (have no area in $\M$).  
Discarding those, there remain $(2N)(2N-2)$ nondegenerate quadrilaterals, which will be seen to have the
properties described in the statement above.

It will be helpful to note right away that within a given $Q_{ij}$, each level set in $t$ or $T$ 
is a straight line.  
This is because $\iota$ is affine on each edge.
Thus, within the quadrilaterals and for fixed $T$, the function $\trace$ is
of the form $(t\A+\B)/(c-t)$, so its derivative is $\frac{1}{(c-t)^2}[c\A+\B]$, which has constant direction.
This shows that 
$t$-level sets (and, similarly, $T$-level sets) are a family of straight segments sweeping out each quadrilateral, though not
in general a parallel family.
Let us also consider level sets of $\trace$ at $T-t=k$.  In this case we have 
$$\trace_{t,t+k}(1)=\frac{\iota(t+k)-\iota(t)}k,$$
which produces a parallel family of straight segments
within $Q_{ij}$ indexed by $k$.

We now 
consider paths starting and ending on the $i$th side.  This quadrilateral $Q_{ii}$ has two components:  
either 
the whole trace path is contained in the $i$th side, or
the trace path traverses all the sides of $I$ before returning.
Recall that the side $\sigma_i$ is parallel to a vertex direction from $L$, say for vertex $\V_i$.
If the trace path is totally contained in $\sigma_i$ then the 
image $\trace_{t,T}(1)$ is the vertex $\V_i$ itself, because it points in that direction and has length one
in the $L$ norm.  
In the other case, we point in the $-\V_i$ direction. 
The magnitude of the vectors in this component of $Q_{ii}$
range from 0 (achieved when $T=1+t$, so the path closes up)
to $\ell_i/(1-\ell_i)$
(achieved when $t$ is at the end of $\sigma_i$ and $T$ has wrapped around to the beginning of $\sigma_i$).  
Thus each $Q_{ii}$ is a vertex of $L$ together with a line segment in the interior of $Q$.

Now we will show that $Q_{i,i+1}$ is the edge of $L$ with endpoints $\V_i$, $\V_{i+1}$.
This follows simply from noticing that each of those trace paths is a two-sided path following
direction $\V_i$ for time $a$ and then direction $\V_{i+1}$ for time $1-a$, so that it terminates
at the point $a\V_i+(1-a)\V_{i+1}$.  Notice that these particular trace paths are geodesic in $\R^2$ with respect
to the $L$-norm; 
and indeed since all trace paths have $L$-length one, they are $L$-norm-geodesics
if and only if they terminate on $L$ itself (the sphere of radius one).

For the remaining quadrilaterals, 
there are four extreme points:  we may start at either endpoint 
of $\sigma_i$ and terminate at either endpoint of $\sigma_j$.  
Indeed, there is a quadrilateral traced out by holding the start point or end point on $I_1$ fixed at 
one extreme while
moving the other from extreme to extreme, then alternating which is fixed and which is moving, making
a circuit of length four.  (Its four subpaths are level sets for $t$ or $T$, so they are straight lines.)

Finally, observe that $\trace_{t,t+k}(1)$ is a convex, centrally symmetric polygon for every fixed $k\in(0,1]$.
This polygon varies continuously in $k$, is identically zero when $k=1$ and equals $L$ itself for $k$
sufficiently small.  We can conclude that all points in $Q$ are hit by trace paths.
\end{proof}

\subsection{Beelines}

Let us call a path in $\R^2$ a {\em beeline} if it has $L$-length
one and connects the origin to a point on $L$.  That is, beelines
are  $L$-geodesic segments emanating from the origin.  One can
check that a given path ending interior to a side $\sigma$ of
$L$ is a beeline by verifying that all of its tangent vectors
point towards $\sigma$.  The name
reflects their consistent progress away from the origin, as
opposed to trace paths which typically make turns in several  directions that seem inefficient in the 
$L$-norm.

\begin{lemma}[Existence of beelines]
Every point $(x,y)\in L$ is reached by a unique trace path $\trace_{t,T}$, and this path encloses
a nonnegative area $A=A(x,y)$.  
There are beelines to $(x,y)$ enclosing every signed area in the range $[-A,A]$.  
\end{lemma}

\begin{proof}
If $(x,y)$ is in the edge of $L$ with extreme points $\V_i$ and $\V_{i+1}$, then the trace path to
$(x,y)$
follows two directions, first going distance $a$ in direction $\V_i$ and then distance $(1-a)$
in direction $\V_{i+1}$.
Now vary $s$ from $0$ to $a$ and consider the beeline that goes $s$ in direction $\V_i$
followed by $(1-a)$ in direction $\V_{i+1}$ and finally $(a-s)$ in direction $\V_i$.  This is 
clearly still an $L$-geodesic, since it has length one and terminates on $L$.  But the area it 
encloses with the straight chord between its endpoints varies continuously from $-A$ when $s=0$
to $A$ when $s=a$.
\end{proof}

In general, the full set of beelines to a side with extreme points $\V_i$ and $\V_{i+1}$ 
consists of  parametrizations by arclength of
arbitrary piecewise differentiable paths whose tangent vectors 
lie in the interval of directions between
those two extremes.

\begin{lemma}[Isoarea problem]\label{isoarea}
A curve in $Q$ of $L$-length $1$
has minimal length among all curves with the same endpoints
and balayage area if and only if it is a trace path or a beeline.
\end{lemma}

\begin{proof}
The reverse implication is easy: we have already shown that trace
paths uniquely maximize area to their endpoints.  Since beelines
are $L$-norm geodesics, there is no shorter path between their endpoints.

Now for the forward implication:
Let $\beta$ be a curve in $Q$ of $L$-length $1$ which has minimal
length among all curves with the same endpoints and balayage
area.  If $\beta$ ends on $L$, then it is a beeline by definition.
Otherwise say $\beta$ ends at $(x,y)\in Q^\circ$ and has balayage
area $A'$.  We know that there is a trace path $\gamma=\trace_{t,T}$ with
$\trace_{t,T}(1)=(x,y)$ enclosing area $A(x,y)\ge A'$.
If $A(x,y)=A'$, then $\beta$ is a trace path by Lemma~\ref{dido}.
If not, then we can modify $\gamma$ to be an improvement on $\beta$:  
we know that the region enclosed by $\gamma$ is convex, so we can find a chord of
this body starting and ending on $\gamma$ so that the modified curve $\gamma'$,
which follows $\gamma$ except for a shortcut along this chord, has length $<1$ and encloses area $A'$.
This contradicts the minimality of the length of $\beta$.
\end{proof}

\section{Structure of polygonal \CC metrics}\label{structure}
\subsection{Unit sphere and \CC geodesics}

Returning to the Heisenberg geometry, we find that we have
identified the geodesics, and can thus map the shape of spheres.

Recall that Pansu's theorem tells us that the word metric $(H(\Z),S)$ is asymptotic to the \CC
norm on $H(\R)$ induced by $L$, the boundary of the convex hull of $\pi(S)$.
Note that the number of sides of $L$ is at most $|S|$, but could be smaller if $S$ has several
elements with the same projection to $\M$, or has elements which do not project to vertices of $L$.
If $L$ is a $2N$-gon, then there are $4N^2-4N$ quadrilateral regions $Q_{ij}$ classifying the shapes
of trace paths, while the $2N$ sides of $L$ classify the beelines.

\begin{Thm}[Structure Theorem]
\label{limitshapebypanels}\label{quadratic}  Consider the \CC metric induced 
on $H(\R)$ by a  polygonal norm $\normL\cdot$ on $\M$.
The balayage function $A(x,y)$ is a continuous function over $Q$
whose restriction to each $Q_{ij}$ is a quadratic polynomial in $x,y$.

Let $\z(x,y)$ be the multivalued function given by 
$$\z=\begin{cases} \{\pm A\}, & (x,y)\in Q^\circ \\ [-A,A], & (x,y)\in L. \end{cases}$$
Then the graph of $\z$  is precisely the unit sphere $\eS$ in the \CC metric.

Equivalently: the full set of \CC geodesics of length one based at the origin is identical to 
the set of admissible lifts of trace paths and beelines (and their reflections in the origin).
\end{Thm}

\begin{proof}
For points $(x,y,z)\in H(\R)$ with $z\ge 0$ and $(x,y)\in Q^\circ$,
saying that an admissible path $\alpha$ from $\zero$ to $(x,y,z)$ is geodesic
means that $\pi(\alpha)$ is the shortest path in $\M$ from $(0,0)$
to $(x,y)$ enclosing (nonnegative) balayage area $z$.
By central symmetry of the isoperimetrix, we can solve the problem for minimal signed area
by traversing the isoperimetrix clockwise rather than counterclockwise.  Therefore
 the minimal
signed area enclosed by a path of length 1 from the origin to
$(x,y)$ is given by  $-A(-x,-y)=-A(x,y)$, which tells us that the graph of $\z$ is symmetric over the $xy$ plane.

Continuity of $A(x,y)$ on $Q$ is 
is immediate from the continuity of $\trace(t,T)$ in its parameters.
It only remains to show that the restriction of $A(x,y)$ to each $Q_{ij}$ is a 
quadratic polynomial.

Observe that the problem of finding 
the trace path to a point $(x,y)\in Q_{ij}$ amounts to finding
scalars $s_1,s,s_2$ (with $s_1,s_2\le s$) such that
$$(x,y)=s_1\W_i + s (\W_{i+1}+\cdots+\W_{j-1}) + s_2 \W_j,$$ where $\W_i$ are the vectors
defining the sides of $I$, as in
Figure~\ref{area}. 

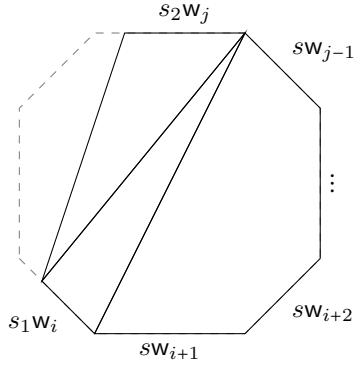
\begin{figure}[ht]
\begin{tikzpicture}
\coordinate (start) at (-1.7,-1.3); 
\coordinate (stop) at (-.6,2);
\draw [dashed,gray] (-2,-1)--(-1,-2)--(1,-2)--(2,-1)--(2,1)--(1,2)--(-1,2)--(-2,1)--cycle;
\draw (-1,-2)-- node [below] {$s\W_{i+1}$} (1,-2)-- 
node [below right] {$s\W_{i+2}$} (2,-1)-- node [right] {$\vdots$} (2,1)-- node [above right] {$s\W_{j-1}$} (1,2)--cycle;
\draw (start)-- node [below left] {$s_1\W_i$} (-1,-2)--(1,2)--cycle;
\draw (1,2)-- node [above] {$s_2\W_j$}  (stop)--(start)--cycle;
\end{tikzpicture}
\caption{The balayage area of a path that follows along a scaled copy of the isoperimetrix
is given by a quadratic expression in $s_1,s_2,s$. \label{area}}
\end{figure}

The area enclosed is equal
to that of a polygon of fixed shape with $j-i$ sides of length proportional to $s$,
so that its area is proportional to $s^2$,
plus two triangles of area proportional to $s_1\cdot s$ and $s_2\cdot s$, respectively.
There are three linear equations relating $s_1,s_2,s$, given by the total arclength equaling one
and the Euclidean displacement between endpoints equaling $x$ and $y$ in the  
horizontal and vertical direction, respectively.  Thus we can solve to get expressions for $s_1,s_2,s$ that are 
linear in $x,y$; this means that the area is quadratic in $x,y$, as required.
\end{proof}

By the definition of a \CC metric, the sphere of radius $t$ based at $\zero$ is precisely $\delta_t\eS$
(so $\eS$ is homeomorphic to a $2$-sphere), 
and the spheres 
with other centers are obtained by pushing these spheres around by left-multiplication in $H(\R)$.

We want to break down the sphere, which is the graph of $\z$ over $Q$, into the graphs over the 
various quadrilaterals $Q_{ij}$.  We will call these the {\em panels} of the sphere, and define them by 
$$\panel_{ij}:=\left\{(x,y,z) : \quad (x,y)\in Q_{ij}, \ \ z=\pm A(x,y)\right\}, \qquad j\neq i+1,$$
and $$\panel_{i,i+1}:=\left\{ (x,y,z) : \quad (x,y)\in Q_{i,i+1}, \ \  z\in [-A,A]\right\},$$
so that $\eS = \bigcup_{1\le i,j \le 2N} \panel_{ij}$.
Note that we can leave out the $\panel_{ii}$
without changing this equality, 
because they are contained in the union
of the other panels.
The $\panel_{i,i+1}$ project to edges of $L$, and the fact that $\z$ is quadratic
over the interior implies that they are bounded above and below 
by parabolas in a plane through that edge that is 
perpendicular to $\M$.  We will call these the \textit{side panels}
of $\eS$ (these are the panels cut away in Figure~\ref{twoviews}).

The combinatorics of trace paths on quadrilaterals dictates the shape of \CC geodesics in the 
following straightforward way.  Let the {\em footprint map} $\Foot: H(\R)\to Q$ be defined by 
$$\Foot(\X) :=  \pi \circ \delta_{\nicefrac 1{d}} (\X)=\delta_{\nicefrac 1d}\circ\pi(\X),$$
where $d=\dCC(\X,\zero).$  Note that all geodesics from $\zero$ to any point are just dilations 
of geodesics from $\zero$ to $\eS$.  These in turn are lifts of trace paths and beelines, whose 
shape is determined by which quadrilateral contains the projection from $\eS$ to $Q$.  
That is, the shapes of all possible geodesics from $\zero$ to $\X$ are determined by which quadrilateral
contains $\Foot(\X)$.  

Let the {\em regular points} in the unit sphere be the union of
the panels over the quadrilaterals interior to $Q$
and the {\em unstable points}
be those in side panels.
$$\Sreg= \bigcup_{j\neq i+1} \panel_{ij} \quad ;
\qquad \Suns = \bigcup \panel_{i,i+1} \ ,$$
so that $\eS=\Sreg\cup\Suns$.
Recall that the points of $Q^\circ$ are reached only by trace paths, while points of $L$ are 
reached by families of beelines.
We have $$\Delta \Sreg = \overline{\Foot^{-1}(Q^\circ)} \quad ; \qquad \Delta \Suns = \Foot^{-1}(L).$$
That is, the regular/unstable distinction divides the space $H(\R)$ according to which of the two
types of geodesic reaches each point.

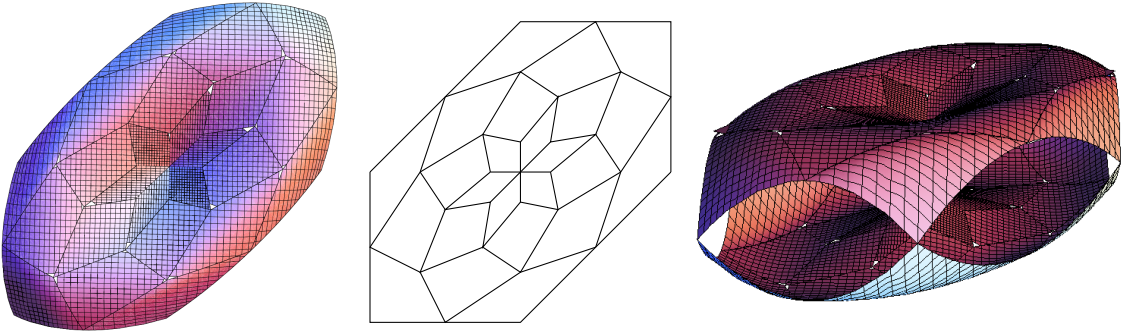
\begin{figure}[ht]
\mbox{\centerline{\includegraphics[width=2.4in]{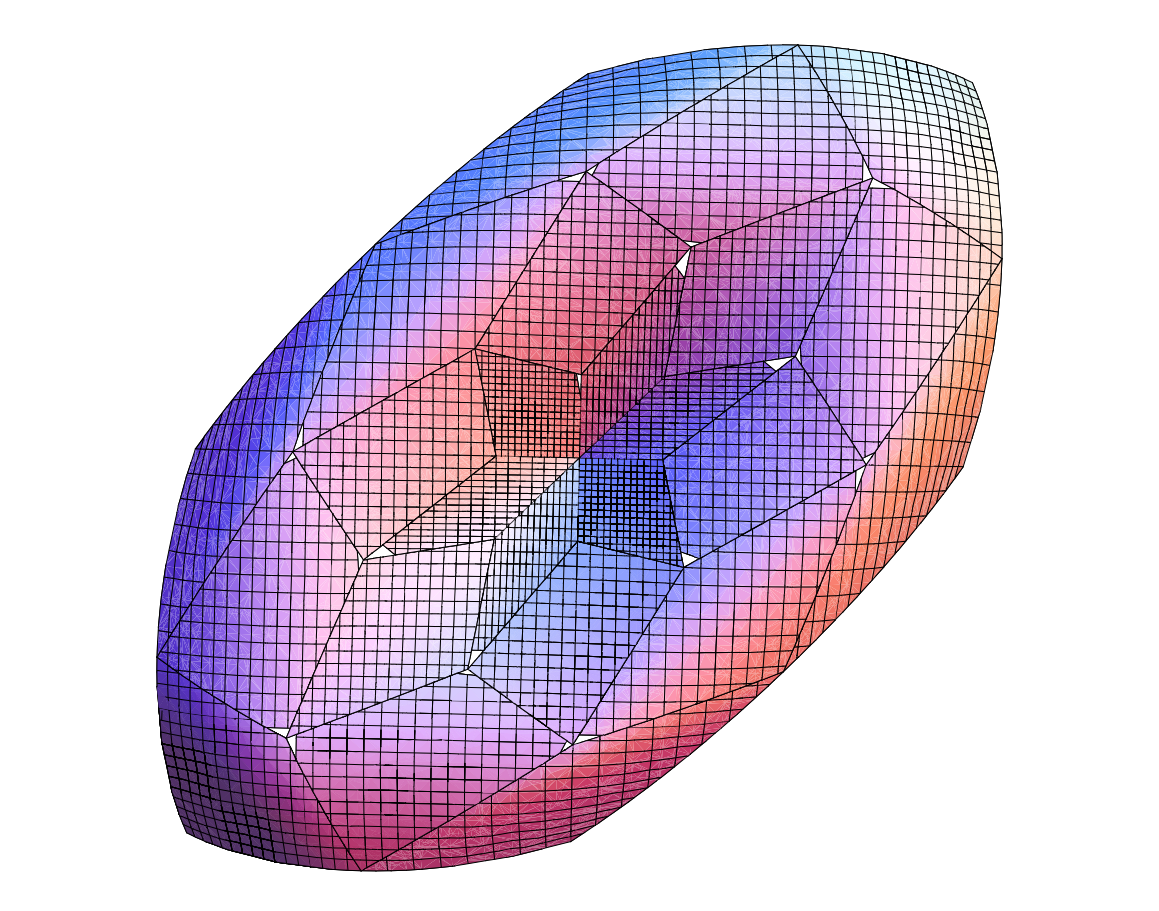}
\hspace{-.2in}\raisebox{.12in}{\begin{tikzpicture}[scale=2]
\foreach \ang in {0,180}
{\begin{scope}[rotate=\ang]
\draw (1,0)--(1,1)--(0,1)--(-1,0);
\draw (-1/2,1/2)--(0,2/3)--(1/2,1)--(2/3,2/3)--(1,1/2)--(2/3,0)--(1/2,-1/2);
\draw (-2/3,0)--(-1/4,1/4)--(0,2/3)--(1/4,1/2)--(2/3,2/3)--(1/2,1/4)--(2/3,0);
\draw (-1/5,0)--(-1/4,1/4)--(0,1/5)--(1/4,1/2)--(1/5,1/5)--(1/2,1/4)--(1/5,0);
\draw (0,1/5)--(0,0)--(1/5,1/5);
\draw (0,0)--(1/5,0);
\end{scope}}
\end{tikzpicture}}
\includegraphics[width=2.5in]{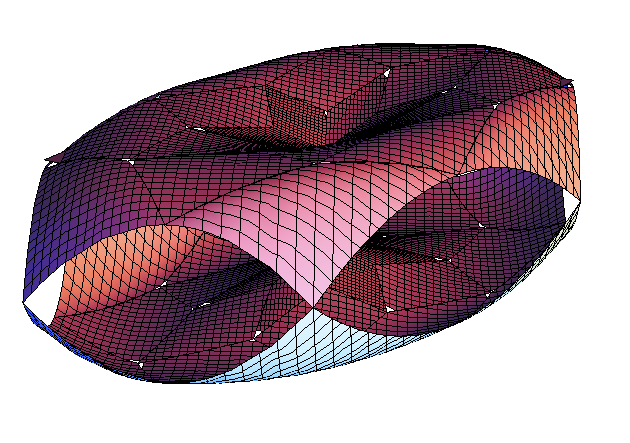}}}
\caption{Two views of the limit shape $\eS$ for the hexagonal generators, with the quadrilateral decomposition 
of the footprint $Q$ reproduced for reference.
Only the regular part, $\Sreg$, is shown.  The unstable part, $\Suns$, is cut away in the picture:
it is made up of six flat panels perpendicular to the $xy$-plane 
that are bounded by parabolas.\label{twoviews}}
\end{figure}

Let the {\em volume subtended by a panel} be the volume of the region obtained by 
coning off to the origin by dilation.
We will see that the the decomposition of the sphere into regular and unstable points
gives us useful invariants coming from volume:
$$\Vreg=\vol(\Sreghat) \quad ; \qquad \Vuns=\vol(\Sunshat).$$
Recalling that $\VS$ is the volume of the unit ball, we have $\VS=\Vreg+\Vuns$.

\begin{remark}[Rationality]
We note that Breuillard carried out a full description for the standard generators of the kind
given in this section,
and indicated key elements of such a description for general $S$, in \cite[Prop 9.1]{breuillard}.
For instance, it is stated without argument there that the balayage function should be piecewise quadratic.  
Breuillard sketches an argument that $\VS$ (the volume of the unit ball for any
limit metric) is rational,
noting that $Q$ has integer vertices, so  its polar dual and therefore 
the isoperimetrix $I$ must have rational vertices (since $Q^*=\{\X\in\M : \ \X\cdot\Y\le 1 \quad  \forall \Y\in Q\}$).
Following this line in our language, we see that
the vectors $\W_i$ have rational projections to each coordinate direction as well as rational
length in the $L$-norm (though not in the Euclidean norm), and so the 
linear relations between $x,y$ and $s_1,s_2,s$ described in the
proof of Theorem~\ref{limitshapebypanels}
are rational, which ensures that the 
balayage area over $Q_{ij}$ 
is a rational quadratic polynomial in $x$ and $y$.  Furthermore, the vertices 
of the quadrilaterals $Q_{ij}$ are rational  (when a trace path begins and ends at a vertex,
and has arclength 1, then each coordinate of the endpoint is a sum of rational numbers, scaled
by a rational number).  Thus the volume subtended by each panel is given by an integral
over a region bounded entirely by graphs of rational quadratic polynomials (both the balayage
area and the tracks of the dilation) over rational polygonal domains, so each such 
volume is in fact rational.
\end{remark}

\begin{remark}[Comparing metrics]
Notice the following fact that easily follows from Pansu's theorem 
from this geometric description of  limit shapes.
For an arbitrary generating set $S$ of $H(\Z)$, despite the fact that the word-length of generators is one by 
definition and the \CC metric is the limit of the rescaled word metrics, it is easily possible for elements of $S$ to 
be artibrarily far from $\zero$ in \CC distance:
one can have $\dCC(\A,\zero)\gg 1=|\A|_S$ for some $\A\in S$, meaning that the generator lies far outside of $\eS$.
This is because $\dCC$ depends only on $\pi(S)$ and not on the $z$-coordinates of the generators.
However, if $\pi(\A)$ is not a vertex of $L$, we will eventually have $\dCC(\A^n,\zero) \le |\A^n|_S$
because the $z$-coordinate of $\A^n$ grows only linearly, while
the height of $\delta_n\eS$ over $\pi(\delta_n\A)$
grows quadratically in $n$.  It is only
over the vertices of $L$ that the height of $\eS$ is zero.  In that
case, we see  that $\dCC(\A^n,\zero)\le |\A^n|_S+1$
for sufficiently large $n$, because the dilation causes the curves defining the sphere near
its corners to get more and more vertical as the radius gets large.
This establishes that there is a constant 
$K=K(S)$ such that
$$\dCC(\X,\zero)  \  \le \   |\X|_S+K$$
for all $\X\in H(\Z)$, which is half of Krat's bounded difference theorem (Thm~\ref{bdd-diff}).
\end{remark}

\subsection{Nonuniqueness and regular points}

The union of the degenerate quadrilaterals has a geometric
significance.  Recall that the quadrilateral $Q_{ii}$ is a line segment based 
at the origin along with a vertex.
Let $\V_i'$ denote the other endpoint of the line segment emanating
from $\zero$.
The other degenerate
quadrilaterals cover the boundary $L$, whose vertices are called $\V_i$.
Then, with the convention that $Q_{2N,2N+1}=Q_{2N,1}$, let
$$\NN_0 :=  \bigcup_{1\le i\le 2N} Q_{ii} \setminus \{\V_i'\}
\quad ; \qquad L_0:=\bigcup_{1\le i\le 2N} Q_{i,i+1} \setminus \{\V_i\}=L\setminus \{\V_i\}.$$ 
Now define $\NN:=\NN_0\cup L_0$.

\begin{prop}[Uniqueness of \CC geodesics]\label{nonunique-geods}
A point  $\X\in\eS$ is reached by more than one geodesic of length
one based at the origin if and only if $\pi(\X)\in\NN$.

Thus a point $\X\in H(\R)$ is reached by more than one geodesic
based at the origin if and only if $\Foot(\X)\in\NN$.
\end{prop}

\begin{proof}
First, we consider the case of trace paths.  
It is impossible for a convex polygon (in this case $\lambda I_1$)
to have two different interior chords which are parallel, have the same 
length, and subtend the same perimeter on each side.  Thus the
only way this non-rigidity can occur is for the chord $\alpha$ to
be contained in a side of $\lambda I_1$.  But this is precisely the 
case when the start and end points of $\trace$ fall on the same side,
which was examined in the discussion of $Q_{ii}$.  The vertices of
$L$ have trace paths enclosing zero area, so the geodesics are
unique in that case.  Finally, if $\pi(\X)$ is interior to a side
of $L$, then geodesics are lifts of beelines that are not straight lines.
These can always be perturbed
to nearby curves with the same endpoints, length, and enclosed area.
\end{proof}

From now on we call $\NN$ the {\em nonuniqueness locus} in $Q$.

\begin{Thm}[Probability of unique geodesics]
In $H(\R)$ with the \CC metric induced by a polygon $L$, 
the asymptotic density of unique geodesics equals $\Vreg/\VS$.

That is, choose $\X$ uniformly in the ball of radius $r$ and consider the probability
that there is only one geodesic from $\zero$ to $\X$.  For every $r$, this probability
equals $\Vreg/\VS$.
\end{Thm}

\begin{proof}
The only instances of nonuniqueness occur over $\NN$, as we have seen. 
But $\Foot^{-1}(\NN_0)$ has no volume, and none of the volume of $\Delta\Suns$ is contributed
by the vertices.
Thus the probability of being reached by distinct geodesics 
is precisely the proportion of the volume of the unit ball that is in the unstable part
(i.e., subtended by the side panels).
\end{proof}

\begin{example}[Volume calculations in $\ell^1$ norm]
Let $\|\cdot\|_L$ be the $\ell^1$ norm on $\M$ and give $H(\R)$
the corresponding \CC metric.
We find that the volume subtended by panels with  4-sided combinatorics is $13/216$, the volume 
subtended by  
3-sided panels is $11/54$, and the volume subtended by the side panels is $1/6$. 
Thus $\Vreg=13/216+11/54= 19/72$ and $\Vuns=1/6$.
This adds up to 
give the total volume of the unit ball as 
$\VS=31/72$, which agrees with calculations by Breuillard \cite{breuillard}
 and Stoll  \cite{stoll1} of the volume growth. 

Thus for this standard polygonal \CC metric, the density of points reached by nonunique geodesics is 
$12/31$, or about $39\%$, versus $19/31\approx 61\%$ for unique geodesics.
\end{example}

\section{Geodesics in the word metric}\label{instability}

\subsection{Comparison of word and \CC geodesics}

Recall that the \textit{Hausdorff distance} $\HausD(\alpha,\beta)$
between a pair of sets is the smallest $\epsilon\ge 0$ such that
each set lies in the $\epsilon$-neighborhood of the other.  Below, we will talk
about paths and the images of those paths interchangeably, to make sense 
of the Hausdorff distance between paths.
The goal of this section is make precise the following statement:
``\CC geodesics approximate word geodesics."

\begin{lemma}[Continuity of \CC geodesics]
\label{cont-geod}
Fix an arbitrary $\rho>0$.
Suppose we are given $\X\in\eS$ such that $\pi(\X) \notin \eN_\rho(\NN_0)$
and a \CC geodesic $\alpha$ from $\zero$ to $\X$.
Then for every $\epsilon>0$ there exists  $\delta>0$
such that whenever $\Y\in\eS$ with $\dCC(\X,\Y)<\delta$,
there exists a \CC geodesic $\beta$ from $\zero$ to $\Y$
with $\HausD(\alpha,\beta)<\epsilon$.
\end{lemma}

\begin{proof}
This is clear for points interior to any regular panel.
For points $\X$ in the boundaries of regular panels, one can see
from the combinatorial description that a sequence of
points approaching $\X$ from any panel must have geodesics
approaching the unique geodesic to $\X$ (see Figure~\ref{combx}).
Suppose $\X$ is in the interior of a side panel and $\Y$
is close to $\X$.  If $\pi(\X)=\pi(\Y)$, then homotope $\pi(\alpha)$
to a path $\overline{\beta}$ with the same endpoints such that
the tracks of individual points are small.
As long as every
tangent vector of every point on every intermediate path points
towards the same side, this is a homotopy
through beelines.  If $\pi(\X)\neq\pi(\Y)$,
then we can concatenate this homotopy with another similar
homotopy ending at a path from the origin to $\pi(\Y)$
without changing area much.  Either way, lift $\overline{\beta}$
to a geodesic $\beta$ ending at $\Y$ which is close to $\alpha$.
\end{proof}

This continuity statement is false if $\pi(\X)\in \NN_0$.  For instance,
take $\X$ to be the point on $\eS$ above $\zero$ and $\{\Y_n\}$ to be
a sequence of points interior to a panel, converging to $\X$.
Then for each $n$, there is a unique geodesic $\gamma_n$ from
$\zero$ to $\Y_n$; these converge to a particular one of the many geodesics
from $\zero$ to $\X$, which include 
the lifts of $\trace_{t,t+1}$ for all $t\in[0,1]$.  These other geodesics to $\X$ are thus
not closely approximated by geodesics to the $\Y_n$ even as $\Y_n\to\X$.

Let $S$ be a finite generating set for $H(\Z)$, and $K$ be the
maximum of $\dCC(\zero,\A)$ over all generators $\A\in S$.
Minimal-length spellings $\W=\A_1\cdots\A_n$ written in letters from $S$
may be realized as {\em discrete geodesics} $\X_0,\X_1,...,\X_n$ in $H(\R)$
where $\X_0=\zero$, $\X_i=\A_1\cdots\A_i$, and the \CC
distance between successive points is no more than $K$.
We will say that the discrete geodesic $\X_0,\X_1,...,\X_n$ is {\em $\epsilon$-linearly tracked}
(or just $\epsilon$-tracked)
by an admissible path $\alpha$
if they have the same endpoints and
the set $\{\X_i\}$ stays inside the $\epsilon n$-neighborhood
of $\alpha$.

\begin{lemma}[Geodesic spellings vs. admissible paths]
\label{geodesic spellings vs admissible paths}\label{spellings}
For any $\epsilon>0$ and $\X\in H(\Z)$ with sufficiently large $n=|x|$,
every discrete geodesic from $\zero$ to $\X$ is $\epsilon$-linearly tracked
by an admissible path of length $n+O(\sqrt n)$.
\end{lemma}

\begin{proof}
Let $z(\X)$ denote the height (i.e., the $z$-coordinate) of
$\X\in H(\R)$, and $M:=\max z(S)$.
Let $\X_0,...,\X_n$ be a discrete geodesic  from $\zero$ to $\X\in H(\R)$ in letters from $S$.
Define a polygonal path $\gamma$ in $\M$ by connecting
vertices $\pi(\X_0),\ldots,\pi(\X_n)$ with straight lines.  

The curve $\gamma$ has a unique admissible lift
$\overline\gamma$; its  endpoints are at $\zero$ and at a point $\Y\in H(\R)$
which lies in the same vertical line as $\X$.
Since $\pi(S)\subset Q$, the curve $\gamma$ lies in $nQ$, and since it is made
up of straight segments of $L$-length $\le 1$, the length of
$\gamma$ (and hence $\overline\gamma$) is bounded above by $n$.
We will establish the following equation:
\begin{equation}\tag{$\dagger$}\label{dagger}
	z(\X)=z(\Y)+\sum_{i=1}^nz(\A_i)
\end{equation}
This says that the height of $\X$ is the balayage area of
$\gamma$ plus the heights of all of the letters used in the spelling.

Equation~\eqref{dagger} can be proven by induction.
The group law, in exponential coordinates, says
$$(x,y,z)(x',y',z')
	=\left(x+x', \ y+y', \ z+z'+\frac{xy'-yx'}{2}\right).$$
In particular, if $\X_{n-1}=(x,y,z)$ and $\A_n=(x',y',z')$,
then we see that the change in height between $\X_{n-1}$ and $\X_n=\X_{n-1}\A_n$ is equal to
the height of $\A_n$ plus the  area of an appropriate triangle in $\M$.
That means that the height of the spelling path differs from the height 
predicted by balayage area by precisely the height of the generators,
as claimed in \eqref{dagger}.

Now set $C=\dCC((0,0,M),\zero)$, so that 
$$\dCC(\zero,(0,0,Mn))=\dCC(\zero,\delta_{\sqrt n}(0,0,M))=C\sqrt n.$$

Let $\alpha$ be a path obtained by concatenating $\overline\gamma$ with
a \CC geodesic from $\Y$ to $\X$.  From \eqref{dagger} it follows that
$$|z(\X_i)-z(\Y_i)|\le Mi\le Mn,$$
for all $i$, where $\Y_i$ is the point on $\overline\gamma$
corresponding to the vertex $\X_i$.
Therefore $\dCC(\X_i,\Y_i)\le C\sqrt n$, so the discrete geodesic
lives in the $C\sqrt n$-neighborhood
of $\alpha$, which means that the path $Cn^{-1/2}$-tracks the spelling.
Furthermore, the difference in lengths between $\alpha$
and $\overline\gamma$ is bounded by $C\sqrt n$.
Thus the length of $\alpha$ is bounded above by $n+C\sqrt n$.
On the other hand, Theorem~\ref{bdd-diff} provided a global
constant $C'$ such \(\dCC(\zero,\X)\ge n-C'\).  This proves that
the length of $\alpha$ is $n+O(\sqrt n)$.
\end{proof}

\begin{lemma}[Tracking Lemma]
\label{tracking}
For every $\epsilon>0$, $\rho>0$, and sufficiently large $\X\in H(\Z)$,
if $\Foot(\X) \notin \eN_\rho(\NN_0)$, then every discrete geodesic
from $\zero$ to $\X$ is $\epsilon$-linearly tracked by a \CC geodesic.
\end{lemma}

\begin{proof}
Suppose otherwise; then there is an $\epsilon>0$ such that for
every $n\ge 0$, there exist points $\X_n\in H(\Z)$ such that
$\Foot(\X_n)  \notin \eN_\rho(\NN_0)$, $|\X_n|\ge n$
and there are choices of discrete geodesics from $\zero$ to $\X_n$
which stay $\epsilon|\X_n|$-far away from any \CC geodesics
to the same point.  Let $\alpha_n$ be the admissible paths
approximating these discrete geodesics, as described above.

Denote $\overline{\X}_n=\delta_{1/|\X_n|}\X_n$
and $\overline{\alpha}_n=\delta_{1/|\X_n|}\alpha_n$.
The lengths of the $\overline{\alpha}_n$ are converging to 1 by Lemma~\ref{spellings}.
By passing to a subsequence,
we may assume that the $\overline{\X}_n$ converge to a limit
$\overline{\X}\in\eS$ with $\pi(\overline\X) \notin \eN_\rho(\NN_0)$, and that
the paths $\overline{\alpha}_n$ converge in Hausdorff distance to a geodesic $\overline{\alpha}$
from $\zero$ to $\overline{\X}$ by Arzel\`a-Ascoli.  
By Lemma \ref{cont-geod},
there is a large enough $n$ such that there is a geodesic
$\overline{\beta}_n$ within a Hausdorff distance of $\epsilon/2$
from $\overline{\alpha}$,  and for large enough $n$ we know that $\overline\alpha_n$
is within  $\epsilon/2$ of $\overline\alpha$.
But then for $\beta_n=\delta_{|\X_n|}\overline{\beta}_n$, we have
$\HausD(\alpha_n,\beta_n)<\epsilon|\X_n|$.  So the discrete geodesics
are $\epsilon$-tracked by the \CC geodesics $\beta_n$,
giving us a contradiction.
\end{proof}

A key idea in above proof is that 
additive quasi-geodesics get close to true geodesics as the additive constant goes to zero.
Again, this fails near $\NN_0$.

\begin{figure}[ht]
\begin{tikzpicture}[scale=.2]
\def\n{3}
\filldraw (0,0) circle (0.2) node [left] {$(0,0)$};
\filldraw (\n,\n) circle (0.2) node [above left] {$(n,n)$};
\draw (0,0)--(0,-\n*\n + \n) -- node [below] {$n^2$} (\n*\n,\n-\n*\n)-- node [right] {$n^2$} (\n*\n,\n)--(\n,\n);
\draw [dashed] (0,0)--(\n,\n);
\begin{scope}[xshift=20cm]
\filldraw (0,0) circle (0.2) node [left] {$(0,0)$};
\filldraw (\n,\n) circle (0.2) node [above left] {$(n,n)$};
\draw (0,0)-- node [below] {$n^2$} (\n*\n,0)-- node [right] {$n^2+n$}  (\n*\n,\n*\n+\n)--(\n,\n+\n*\n)--(\n,\n);
\draw [dashed] (0,0)--(\n,\n);
\end{scope}

\end{tikzpicture}
\caption{In these diagrams in $(\M,\ell^1)$, we see a true geodesic $\alpha_n$ 
from $(0,0)$ to $(n,n)$ on the left and a path $\beta_n$
with different combinatorics, enclosing the same area but with excess length, on the right.
When these are rescaled so that their lifts 
terminate on the sphere of radius $1$, the ratio of lengths goes to $1$
while the Hausdorff distance is bounded away from zero.
\label{nonlimit}}
\end{figure}
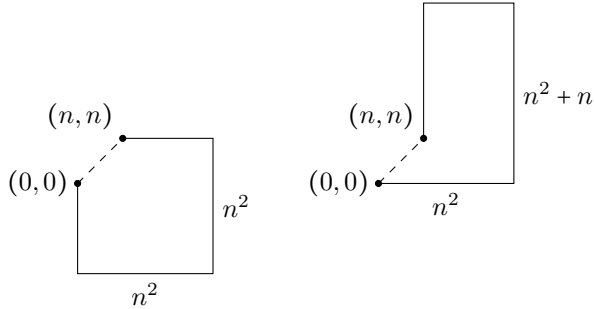

\begin{example}
Let $\X_n=(n,n, \ n^4-\frac 12 n^2)$.
There is a unique \CC  
geodesic $\alpha_n$ from the origin to $\X_n$, and its projection to the horizontal plane is 
shown in Figure~\ref{nonlimit} on the left.  On the right is pictured a non-geodesic path $\beta_n$
to $\X_n$
which has length $4n^2$ (while the true geodesic has length $L_n=4n^2-2n$).  
Let $\Y_n = \delta_{1/L_n}\X_n$ be the dilate of $\X_n$ that lies in the unit sphere $\eS$.
Note that $\Y_n\to (0,0,\nicefrac1{16})$ as $n\to\infty$.
By scaling back the paths $\beta_n$ by the same factor, we obtain paths whose lengths
approach $1$.  However, the Hausdorff distance between $\alpha_n$ and $\beta_n$ is $n^2-n$,
so after scaling they still have a Hausdorff distance of at least $1/4$.  In the limit as $n\to\infty$,
both kinds of paths limit to valid closed geodesics to $(0,0,\nicefrac1{16})$.  But for any finite value of $n$, 
the path with the wrong combinatorics is far from any geodesic.
\end{example}


\subsection{Instability of geodesics in $H(\R)$ and $H(\Z)$}

\begin{defn} For a geodesic space $X$ with basepoint $x_0$,
let $\G(x)$ be the set of all geodesics from $x_0$ to $x$.
Then the {\em instability function} $\I(x)$ measures how far apart they can be, relative
to the size of $x$:
$$ \I(x) := 2\sdot\frac{\diam\G(x) }{d(x,x_0)} = 
{\textstyle\frac{2}{d(x,x_0)}}\sdot\sup_{\alpha,\beta\in\G(x)} \HausD(\alpha,\beta),$$
where the metric on $\G(x)$ is given by Hausdorff distance.

We will say that a point $x$ is {\em $\epsilon$-stable} if $\I(x)\le \epsilon$. 
\end{defn}

In general, $\diam\G(x) \le d(x,x_0)/2$, because any point on a geodesic from $x_0$ to $x$ is 
within $d(x,x_0)/2$ of one of the 
endpoints, and therefore within that distance of any other geodesic.  
Thus the factor of $2$ in the definition normalizes $\I$ so that $0\le \I(x) \le 1$  for any $X,x_0,x$.   

Note that if $X$ has unique geodesics, then $\I\equiv 0$.
If $X$ is $\delta$-hyperbolic, then $\diam\G(x)\le \delta$ for all $x$, so 
$\I(x)\to 0$ as $x\to\infty$.  

The asymptotic behavior of this instability function is far from being quasi-isometry
invariant.  Obviously $\I\equiv 0$ on the Euclidean plane $\R^2$, but 
the situation is quite different for $\R^2$ with a polygonal norm.  For instance, in the $\ell^1$ norm, 
the elements along the $x$ and $y$ axes have $\I=0$, but the diagonal has $\I(a,a)=1$.  
It is easily verified that with any generating set, the limit shape $L$ divides the plane into sectors,
and the instability is $0$ on the boundaries of the sectors and $1$ on the midlines, varying linearly
between these extremes
along the sides of $L$.  In particular one can recover the vertex directions of $L$ from the function $\I$.
Since the word metric on $(\Z^2,S)$ is asymptotic to $\M_L=(\R^2,\normL\cdot)$, it is
also true for word metrics that $\I$ detects the directions of significant generators.

The goal of this section is to study instability in the Heisenberg
group.  As one would expect,  $H(\Z)$ is more geodesically stable than $\Z^2$
but less stable than a hyperbolic group.  More intriguingly,  the stability depends on 
the generating set in a different way than for those other groups.


We will call an element $H(\Z)$ regular if it is regular as an element
of $H(\R)$.

\begin{lemma}[Stability of regular points]
For every $\epsilon>0$ and $\rho>0$, all sufficiently large
regular points $\X\in H(\Z)$ with $\Foot(\X) \notin \eN_\rho(\NN_0)$ satisfy
$$\HausD(g,h)<\epsilon|\X|$$ for all word geodesics $g,h$ from
$\zero$ to $\X$.
\end{lemma}

\begin{proof}
Let $S$ be a finite generating set for $H(\Z)$
and give $H(\R)$ the corresponding \CC metric.
Let $T\ge 0$ be large enough so that when $|\X|\ge T$ and
$\Foot(\X) \notin \eN_\rho(\NN_0)$ and $g$ is a word geodesic from $\zero$ to
$\X$, then there is a \CC geodesic $\gamma$ from $\zero$ to $\X$
such that $\HausD(g,\gamma)<\frac\epsilon2|g|=\frac{\epsilon}2 |\X|$.
Since regular points have unique \CC geodesics, we are done.
\end{proof}

Then since the proportion of $\Delta\Sreg$ lying over $\eN_\rho(\NN_0)$ goes to zero
as $\rho\to0$, we can conclude from the point of view of geometric probability that
all of the regular part of $H(\R)$ is $\epsilon$-linearly stable for every $\epsilon$, so that
Hausdorff distance between word geodesics is bounded above by 
{\em every} linear function of word length.
That is, for regular points of $H(\Z)$,
word geodesics satisfy a sublinear fellow traveling property.
This is much stronger control than one has in $\Z^d$, where
 only words that are powers of single generators can be
$\epsilon$-stable for every $\epsilon$, so almost all points are unstable.
At the other extreme, all elements sufficiently far from the basepoint in a hyperbolic
group $G$  are stable.
The Heisenberg group is intermediate between these in terms of geodesic stability:
only the regular part is $\epsilon$-stable for all $\epsilon$, and its measure
depends nontrivially on the generators.

To sum up this comparison:
\begin{Thm}[Instability Theorem] \label{instab}
With respect to any finite generating set, 
$$\lim_{\epsilon\to 0} \Prob\bigl(\I(\X)\le \epsilon\bigr)=
\begin{cases} 
0, & X=\Z^d,\\
{\Vreg}/{\VS}, & X=H(\Z),\\
1, & X ~\hbox{\rm unbounded,}~\delta\hbox{\rm -hyperbolic.}
\end{cases}$$
\end{Thm}

%
%

\section{Lattice point counting}\label{counting}

Because the balls in the \CC metric get ``fat" (i.e., the volume of a $1$-neighborhood of the sphere is lower-order
than the volume of the ball), it is easy to see that the number of lattice points in a ball of large radius  
equals the volume of the ball (to first order).  What is much more delicate is to count lattice points in an annular region 
between a sphere of radius $n$ and a sphere of radius $n-1$.

As before, the footprint of the limit shape, $Q$, is covered by finitely many quadrilaterals $Q_{ij}$.
For a subset $U$ of $H(\R)$, let $\Delta_n U = \{ \delta_t(U) : n-1 < t \le n\}$ 
be the ``annular" region between
the $(n-1)$st and $n$th dilate of $U$.

Recall the functions $\z$ and $A$ described in Theorem
\ref{limitshapebypanels}.  These functions are related by $A=\max\z$;
that is, $A$ is the height (maximum $z$-coordinate) of the unit sphere in the \CC metric
over points in $Q$, or equivalently, it is the maximum balayage area of curves of $L$-length $1$
from the origin to $(x,y)$.
We can compute the height of the \CC sphere of radius $n$ above
$(x,y)\in nQ$ using the rule $H_n(x,y) = n^2 \sdot A(\nicefrac xn,\nicefrac yn)$.
(Note that for fixed $(x,y)$ and large enough $n$,
this is continuous and strictly increasing in $n$.)
The unit ball is compact, so $A$ takes a maximum value
$\Amax$ on $Q$ (the largest height achieved by the unit sphere).
Then for fixed $(x,y)$, the height function $H_n(x,y)$ is
bounded above by $\Amax\cdot n^2$.   But this is not by itself good enough
to get control on the height difference between the $n$-sphere 
and the $(n-1)$-sphere over a point, $h_n:=H_n-H_{n-1}$.

\begin{lemma}[Control on height]
Take a polygonal \CC metric on $H(\R)$ and
let $h_n(x,y)$ be the height difference between the spheres of
radius $n$ and $(n-1)$ over $(x,y)$ for $(x,y)\in (n-1)Q$.
Fix any nondegenerate quadrilateral $Q_{ij}$ and let
$\Omega_n\subset nQ_{ij}$ denote the complement
of the 1-neighborhood of the boundary.
Then $h_n=O(n)$ over $\Omega_n$.
\end{lemma}

\begin{proof}
First of all, for $\X\in \Omega_n$, one easily verifies that both $\frac1n\X$
and $\frac{1}{n-1}\X$ lie in $Q_{ij}$.
Recall from the structure theorem
(Thm~\ref{quadratic}) that the nonnegative part of 
$\eS$ over $Q_{ij}$ is the graph of a single quadratic
polynomial in $x,y$ ; denote this
polynomial by $f(x,y)=a_1 x^2 +a_2 xy + a_3 y^2 + b_1 x + b_2 y +c$.
For all $(x,y)\in\Omega_n$, we have
\[
	\bigl( x,y, \ n^2 \sdot f(\textstyle \frac xn, \frac yn) \bigr)
	\quad
		\in
	\;
	\delta_n(\panel_{ij})
\]
so
\[
	h_n(x,y)
		=
	n^2\sdot f(x/n,y/n)
		-
	(n-1)^2\sdot f(\nicefrac{x}{n-1},\nicefrac{y}{n-1}) = b_1x+b_2y+2cn-c.
\]
Since $(x,y)\in\Omega_n\subset nQ_{ij}\subset nQ$, both $x$ and $y$ are $O(n)$, and 
therefore $h_n$ is as well.
\end{proof}

\begin{Thm}[Counting Theorem]\label{countingthm}
Consider the \CC metric induced 
on $H(\R)$ by a norm $\normL\cdot$ on $\M$, with $L$ a polygon.
Let $\sigma$ be an arbitrary measurable subset of $\eS$, the \CC unit
sphere.
Then 
$$\#\left( H(\Z) \cap \Delta_n\sigma \right)= 4n^3 \vol(\hat\sigma) + O(n^2).$$
\end{Thm}

\begin{proof} First we consider $\sigma\subset\Sreg$.
For this case, it suffices to treat small subsets $\sigma$
which project to squares in $\M$ whose closures are contained
in the interior of a single quadrilateral:
$$U:=\pi(\sigma)=(u,u+\epsilon)\times (v,v+\epsilon)$$
such that $[u,u+\epsilon]\times[v,v+\epsilon]\subset Q_{ij}^\circ$.
Since the dilation $\delta_t$ is just a homothety on $\M$, the 
dilate $tU$ is a square for every $t$.
The condition on the closure of $U$ guarantees that when $n$
is large enough, $nU$ does not intersect the $1$-neighborhood
of the boundary of $nQ_{ij}$.

Consider the rectangle $U_n:=nU \cap (n-1)U$, let
$A_n:=\Delta_n\sigma \cap \pi^{-1}(U_n)$ be 
the part of the $n$th annular shell that is vertically over
$U_n$, and let $B_n:=\Delta_n\sigma \setminus A_n$ be the rest
(see Figure~\ref{shell anatomy}).

\begin{figure}[ht]
\begin{tikzpicture}
\def\a{.6} \def\b{.85} \def\m{7} \def\M{8} \def\cheat{2.2}
\draw [->](0,-1) -- (0,5);
\draw [->](-1,0)--(10,0);
\draw [ultra thick] (\cheat*\a,0)-- node [below] {$U$} (\cheat*\b,0);
\draw [red] (\m*\a,0)--(\m*\b,0);
\draw [red] (\M*\a,0)--(\m*\b+\b,0);
\draw [ultra thick,blue!80!black] (\M*\a,0)-- node [below,black] {$U_n$} (\m*\b,0);

\coordinate (s) at (\cheat*\a,\cheat*\cheat*\a*\a/7);
\coordinate (t) at (\cheat*\b,\cheat*\cheat*\b*\b/16);
\coordinate (p) at (\m*\a,\m*\m*\a*\a/7);
\coordinate (q) at (\m*\b,\m*\m*\b*\b/16);
\coordinate (x) at (\M*\a,\M*\M*\a*\a/7);
\coordinate (y) at (\M*\b,\M*\M*\b*\b/16);
\draw (0,0) parabola (6,36/7);
\draw (0,0) parabola (8,4);
\begin{scope}
\clip (0,0) parabola (x) parabola (y) parabola bend (0,0) (0,0);
\clip (p) parabola (q)--(10,2)--(10,5)--(\m*\a,5)--cycle;
\draw [fill=red!80] (\m*\a,0) rectangle (\M*\b,5);
\draw [fill=blue!80!black] (\M*\a,0) rectangle (\m*\b,5);
\end{scope}
\draw [very thick] (s) node [above] {$\sigma$} parabola (t) ;
\draw (p) parabola (q);
\draw (x) parabola (y);

\draw [gray] (\m*\a,0)--(\m*\a,5);
\draw [gray] (\M*\a,0)--(\M*\a,5);
\draw [gray] (\m*\b,0)--(\m*\b,5);
\draw [gray] (\M*\b,0)--(\M*\b,5);

\end{tikzpicture}

\caption{This shows  an $xz$-plane slice,
with $U_n$ along the $x$-axis, $A_n$ in blue, and $B_n$ in red.  We have $A_n\cup B_n = \Delta_n\sigma$;
we want to count the lattice points in that entire region.  The $A_n$ lattice points are easier to count, and the number
in $B_n$ is lower-order.}
\label{shell anatomy}
\end{figure}
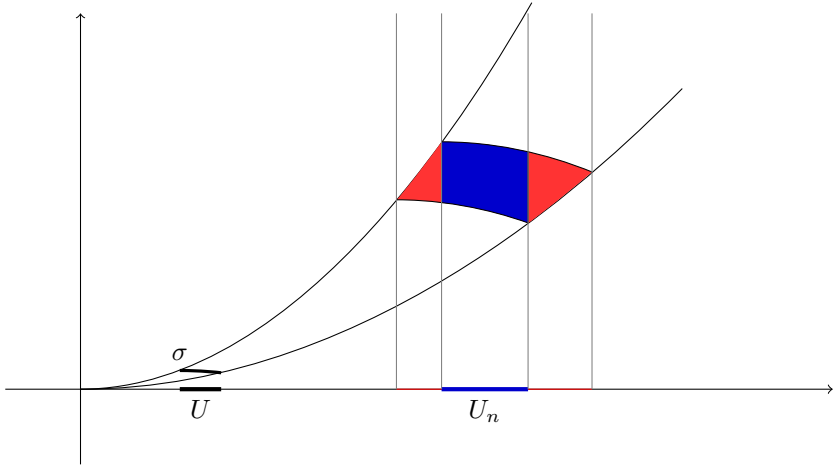

We will make the following chain of comparisons: 
$$\#\left( H(\Z) \cap \Delta_n\sigma \right)\sim \#\left( H(\Z)
\cap A_n \right) \sim \vol(A_n) 
\sim \vol(\Delta_n\sigma) \sim 4n^3 \vol(\hat\sigma),$$
while keeping track of the error term.

First we will estimate the number of lattice points in $B_n$, and
its volume.
The projection $\pi(B_n)$  is contained in a square annulus in $\M$
of width $O(1)$ and sidelength $O(n)$.  Thus the projection contains at most $O(n)$ lattice points $(x,y)$.  
By the previous lemma, the difference in heights between $\delta_n\eS$ and $\delta_{n-1}\eS$ over each of these points is at most $O(n)$, so the
total number of lattice points in $B_n$ is at most $O(n^2)$. 
Thus
$$\#( H(\Z) \cap \Delta_n\sigma) = \#(H(\Z) \cap A_n) + O(n^2).$$
Likewise, the volume of $B_n$ is bounded above by integrating $h_n(x,y)$ over 
$(x,y)\in\pi(B_n)$, so $\vol(B_n)=O(n^2)$.
This shows that $\vol(A_n)=\vol(\Delta_n\sigma) + O(n^2)$.
For the comparison on the far right, we have 
$$\vol(\Delta_n\sigma)=\vol(\delta_n\hat\sigma)-\vol(\delta_{n-1}\hat\sigma)
=\left[ 4n^3+O(n^2) \right] \vol(\hat\sigma),$$
 because $\delta_t$ expands volume by $t^4$.
 
All that remains is to compare $ \#\left( H(\Z) \cap A_n \right)$ with $\vol(A_n)$.
Fix attention on an integer point $(x,y)\in U_n$.  For all the 
lattice points in the vertical line over $(x,y)$, let us 
count the number in $\Delta_n\sigma$.   
We find $ \#(H(\Z) \cap A_n)  = (\sum_{U_n \cap \Z^2} h_n) + O(n^2)$, 
because $\#(U_n\cap\Z^2)=O(n^2)$ and there is bounded error on each line.
On the other hand, $\vol(A_n)  = \int_{U_n} h_n$.
Regard the sum as a Riemann sum approximation to the integral
over the integer squares.  
Since $h(n)$ is linear, the Riemann sum gives an equal error on each of the 
$O(n^2)$ squares, for total error of order $n^2$.

Next we deal with the case $\sigma\subset\Suns$, where it suffices to consider
 $\sigma=\gamma\times[h,h+\epsilon]$ with $\gamma$ a
subinterval of an edge of $L$.  
Let $\Sigma=\Delta\gamma$ denote the sector of $\M$ determined by $\gamma$,
and let $U_n$ be the trapezoid in $\Sigma$ between $n\gamma$ and $(n-1)\gamma$.
We see that $\Delta\sigma= \bigcup_{t\ge 0} t\gamma \times [ht^2,(h+\epsilon)t^2]$.
Let the height difference from top to bottom at $(x,y)$ be denoted $T(x,y)=\epsilon\sdot\normL{(x,y)}$.

We shall make the comparison $$\#(H(\Z)\cap \Delta_n\sigma) \sim \vol(\Delta_n\sigma)$$
by considering the Riemann sum approximation of  $\int_{U_n}T$
by  $\sum_{U_n \cap \Z^2}T$.  For any  $\U,\V\in U_n$, the $L$-norm is between $n-1$ and $n$,
so $T(\V)-T(\U) = \epsilon\normL\V - \epsilon\normL\U \le 2\epsilon n$, which gives a bound on the 
 error per term in the Riemann sum.
Since $U_n$
has bounded width, its area and $\#\Z^2\cap U_n$ are both $O(n)$.
Putting these factors together, we get that the total error is $O(n^2)$,
so we have shown that $\#(H(\Z)\cap \Delta_n\sigma)=4n^3\vol(\hat\sigma)+O(n^2)$, as
in the regular case.
\end{proof}

\section{Limit measure in the standard generators}\label{std-sec}

To get more precise results that allow averaging over spheres and not just balls, 
we focus attention on $H(\Z)$ with the standard generating set $S=\std=\pm \{ \E_1,\E_2\}$ and 
consider the word metric with respect to those generators.  The  goal for this section is to show 
for this special case that counting measure on discrete spheres
in the word metric limits to cone measure (with respect to Heisenberg dilation) 
on the limit shape $\eS$.


\begin{Thm}[Limit measure for standard generators]\label{std}
Consider the \CC metric induced by the $\ell^1$ norm on $\M$, which is the limit metric for $(H(\Z),\std)$.
For any measurable set $\sigma\subset \eS$, 
$$\# ( S_n \cap \Delta \sigma ) = n^3 \vol(\hat \sigma) + O(n^2),$$
and therefore
$$ \lim_{n\to\infty} \frac { \# ( S_n \cap \Delta \sigma ) }{\# S_n} =\frac{\vol(\hat\sigma)}{\vol(\hat\eS)}.$$

Thus for the discrete spheres $S_n$ in $(H(\Z),\std)$, the counting measure on $\delta_{1/n}(S_n)$
converges to cone measure on the square $L\subset\M$.
\end{Thm}

To prove this, we will study the distribution of points in $S_n$, 
finding that the set $S_n$ has points on the line $\{(x,y,t): t\ge 0\}$  if and only if $x,y$ are integers and
$x+y+n$ is even; in that case it contains all lattice points between $\delta_n(\eS)$ and $\delta_{n-2}(\eS)$.
This is illustrated in  Figure~\ref{cross-sec}.

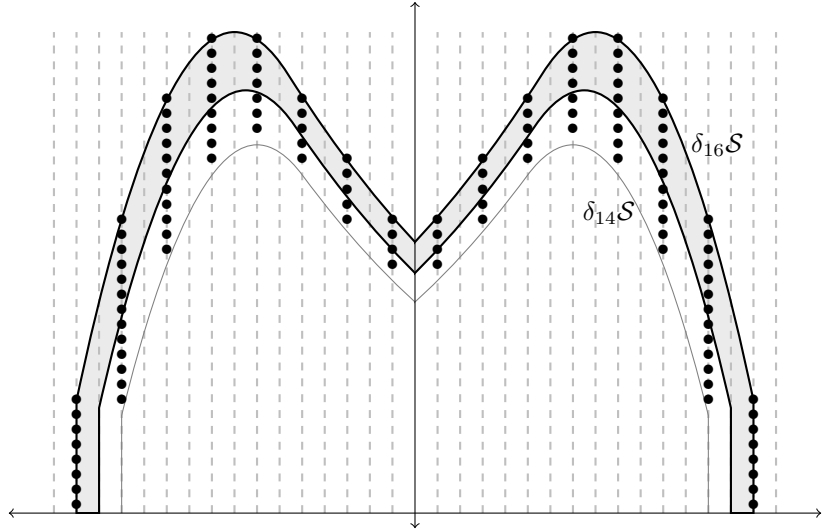
\begin{figure}[ht]
\begin{tikzpicture}[scale=.2]
\begin{scope}[xscale=1.5]
\foreach \flip in {-1,1}
{\begin{scope}[xscale=\flip]
\clip (0,0) rectangle (18,34);
\foreach \x in {1,...,16}
{\draw [thick,gray!50,dashed] (\x,0)--(\x,32);}
\draw [gray]  (-11,13/2) parabola (5,45/2) parabola bend (7,49/2) (13,13/2)--(13,0); 
\draw [thick,fill=gray!50,fill opacity=.3] (-12,7) parabola (16/3,8*29/9) parabola bend (15/2,225/8) (14,7)--(14,0)--
(15,0)--(15,15/2) parabola bend (8,32) (17/3,17*31/18) parabola bend (-13,15/2) (-13,1/4)--cycle;
\foreach \y in {16,...,19}
{\node at (1,\y+.5) {$\bullet$};}
\foreach \y in {19,...,23}
{\node at (3,\y+.5) {$\bullet$};}
\foreach \y in {23,...,27}
{\node at (5,\y+.5) {$\bullet$};}
\foreach \y in {25,...,31}
{\node at (7,\y+.5) {$\bullet$};}
\foreach \y in {23,...,31}
{\node at (9,\y+.5) {$\bullet$};}
\foreach \y in {17,...,27}
{\node at (11,\y+.5) {$\bullet$};}
\foreach \y in {7,...,19}
{\node at (13,\y+.5) {$\bullet$};}
\foreach \y in {0,...,7}
{\node at (15,\y+.5) {$\bullet$};}
\end{scope}
}
\draw [<->] (0,-1)--(0,34); \draw [<->] (-18,0)--(18,0);
\node at (13.4,24.5) {$\delta_{16}\eS$};
\node at (8.6,20) {$\delta_{14}\eS$};
\end{scope}
\end{tikzpicture}
\caption{This figure depicts a cross-section of $H(\R)$, with the points of the sphere $S_{16}$
in $(H(\Z),\std)$ marked and three dilates of the \CC sphere shown for comparison.  
As will be proven in this section,  $S_n$ has very nearly as many points 
as the lattice points contained in $\Delta_n\eS$ (the shaded region):  it has twice the points on half the lines.
\label{cross-sec}}
\end{figure}

\begin{lemma}[\CC distance versus height]\label{height-vs-CC}
For any polygonal \CC metric,  distance from the origin is an increasing function of height:
regarded as a function of $t\ge 0$, the value $\dCC\bigl((x,y,t),\zero\bigr)$ is continuous, 
stays constant on an interval $[0,T(x,y)]$, and is strictly increasing thereafter.  
\end{lemma}

\begin{proof}Let $\X_t=(x,y,t)$.
The distance of a point $\X$ from $\zero$ is the value $d$ such $\delta_{1/d}(\X)\in\eS$.
If $\X_t\in \Delta \Suns$, then the dilate of $\X_t$ that hits $\eS$ will intersect a side panel, so 
the distance of $\X_t=(x,y,t)$ to the origin depends only on $x,y$.
Now $\Delta\Suns$ is a closed set with interior which contains
$(x,y,0)$ for every $(x,y)\in\M$.
Thus $(x,y,t)\in \Delta\Suns$ for a closed interval of times $[0,T(x,y)]$.  

Now consider $\X_t$ and $\X_s$ with $s>t$.  These have the same projection to $\M$, and 
so do $\delta_r(\X_t)$ and $\delta_r(\X_s)$ for any $r$.  Thus if $\delta_{1/d}(\X_t)\in\eS$,
we must have $\delta_{1/d}(\X_s)$ vertically above $\eS$, meaning that 
$\dCC(\X_s,\zero)>d$.
\end{proof}

Note also that if $\alpha$ is a path in a Cayley graph for $H(\Z)$
(i.e., a word written in terms of the  generators), then
there is a corresponding admissible path in $H(\R)$
which ends at the same point and whose \CC length is the same
as the word-length in that word metric.
(See the proof of Lemma~\ref{geodesic spellings vs admissible paths}
for the construction, noting that the extra height created by generators is zero for $S=\std$, 
or whenever the generators lie in $\M$.)
That is, word paths are realizable as admissible paths in an obvious way, as in Figure~\ref{403}.

\begin{figure}[ht]
\begin{tikzpicture}
\filldraw (0,0) circle (.05); \filldraw (4,0) circle (.05);
\draw [dashed] (0,0) -- (4,0);
\draw [very thick] (0,0)--(0,-3/4)--(4,-3/4)--(4,0);
\draw [gray!50] (0,0) rectangle (4,-1);
\foreach \x in {1,2,3}
{\draw [gray!50] (\x,0)--(\x,-1);}
\begin{scope}[xshift=6cm]
\filldraw (0,0) circle (.05); \filldraw (4,0) circle (.05);
\draw [dashed] (0,0) -- (4,0);
\draw [very thick] (0,0)--(0,-1)--(3,-1)--(3,0)--(4,0);
\draw [gray!50] (0,0) rectangle (4,-1);
\foreach \x in {1,2,3}
{\draw [gray!50] (\x,0)--(\x,-1);}
\end{scope}
\end{tikzpicture}
\caption{Comparing geodesics:  $\X=(4,0,3)$ is reached by a three-sided \CC geodesic 
of length $5 \nicefrac 12$, while its word-length in the standard generators is $6$. \label{403}}
\end{figure}
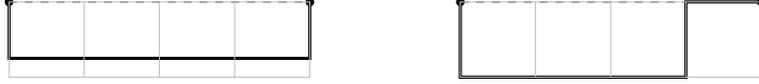

However, the \CC geodesic in $H(\R)$ between two points in $H(\Z)$ is often shorter than the
geodesic in the word metric, as seen in the figure: the \CC geodesics need not have corners at 
integer points, even if they begin and end at integer points.

\begin{lemma}[Word length versus height]
\label{height-vs-word}
For any  $(x,y)\in\Z^2$, let  $\epsilon=\epsilon(x,y)=1/2$ if $x,y$ odd and 
$\epsilon=0$ otherwise, so that $(x,y,\epsilon+m)\in H(\Z)$ for all $m\in\Z$.
Then $|(x,y,\epsilon+m)|_\std$
 is nondecreasing for $m=0,1,2,\ldots$.
\end{lemma}

\begin{proof}Suppose without loss of generality that $(x,y)$ is in the first quadrant of $\M$.
Let  $g$ be a geodesic in $(H(\Z),\std)$ from the origin to a point $\X=(x,y,z)$ where $z>0$.  
Then
$z$ is equal to the signed area enclosed between $g$ and the straight chord from $\zero$ to $\pi(\X)$.
The path contains only $\E_1$ moves and $\E_2$ moves, and there must be some subword $\E_1\E_2$ or else the signed 
area would be nonpositive.  But then let $g'$ be the same path with some $\E_1\E_2$ subword replaced by 
$\E_2\E_1$.  This has the same length and its balayage area is one less, so it is a path in $H(\Z)$ from 
$\zero$ to $(x,y,z-1)$.  This may no longer
be geodesic, but its length provides an upper bound:  $|(x,y,z-1)| \le |(x,y,z)|$.
\end{proof}

Next, it is quite easy to see that there is a parity condition on lengths of paths from the origin to $(x,y,z)$.  
\begin{lemma}[Parity] \label{parity}
For any path of length $n$ in $H(\Z)$ from $0$ to $(x,y,z)$, $n$ has the same parity as $x+y$.
\end{lemma}
\begin{proof}
All paths in $(H(\Z),\std)$ are of the form
$\Pi \E_1^{a_i}\E_2^{b_i}$, with $\sum a_i=x$, $\sum b_i=y$, and  length $n=\sum |a_i|+|b_i|$.
Since $r+|r|$ is even for all integers, it follows that $x+y+n\in 2\Z$.

Geometrically, this is just the observation that if a path in $\M$
to $(x,y)$ goes horizontally past $x$ or vertically past $y$, it
will have to backtrack by the same number of steps.
\end{proof}

For fixed $(x,y)\in\Z^2$ and $n\in\N$, let us count the points $(x,y,z)\in S_n$.
We first note that $\pi(\delta_n\eS)=nQ$, so there are no points of $S_n$ over $(x,y)$ if $\normL{(x,y)}>n$.

\begin{Thm}[Word length in terms of \CC]\label{word-length-formula} Fix $(x,y)\in \Z^2$.  
For $z\ge 0$ in  $\Z+\epsilon$, let $n$ be the unique integer with the same parity as $x+y$
such that $n-2<\dCC((x,y,z),\zero)\le n$.  Then 
$$|(x,y,z)|_\std=n.$$
\end{Thm}

Stated another way, this formula says that for all $\X\in H(\Z)$, 
$$\dCC(\X,\zero)\le |\X|_\std  < \dCC(\X,\zero) +2 ,$$
and describes which of the two integers in that range is the correct value.

\begin{proof}  We will restrict ourselves to the half-space $z\ge 0$ without loss of generality;
the situation on the other half-space is obtained by reflection.
The point $(x,y,\epsilon +m)$ lies in the unstable part of $H(\R)$
for $0\le m\le T$, then in a sector with three-sided combinatorics for some 
$T\le m\le T'$, and finally in a sector with four-sided combinatorics for $m\ge T'$.
(See Figure~\ref{combx}.)
We consider those three cases separately.

Fix  $n\ge 0$ of the same parity as $x+y$.
By Lemma~\ref{height-vs-CC}, there is a unique real value $r\ge 0$ such that $(x,y,r)$ is reached by a 
\CC geodesic of length $n$.  Take that $r$ and suppose that the geodesic $\gamma$ is 
three-sided.  Note that $\pi(\gamma)$ is a path in $\M$ moving only horizontally and vertically;
without loss of generality, 
$\gamma$ has the form $\E_2^{-t}\E_1^x\E_2^{y+t}$, where $t>0$ and $x+2t+y=n$.
But we know $x+y$ has the same parity as $n$, so it follows that
$t$ is an integer.  Thus $\gamma$ is the realization of
a word geodesic, so $|(x,y,r)|=\dCC((x,y,r),\zero)$.

Now assume that $\gamma$ is a 4-sided geodesic from the origin to $(x,y,r)$, with length $n$.
Without loss of generality, $\gamma$ has the form $\E_2^{y-x-t}\E_1^{x+t}\E_2^{x+t}\E_1^{-t}$, where $t\ge 0$ and 
$3x-y+4t=n$.  Since $x+y+n$ is even, it follows that $4t$ is an integer with the same parity as $4x$, so $t$ is either an integer
or a half-integer.  If $t$ is an integer, then $\gamma$ is the realization
of a word geodesic and $|(x,y,r)|=\dCC((x,y,r),\zero)$.  If not, then let $s=t-1/2$ and $u=t+1/2$
be the nearest integers.
We can form an integer path 
$\E_2^{y-x-u}\E_1^{x+s}\E_2^{x+u}\E_1^{-s}$ of length $n$ by shortening the horizontal sides and lengthening the vertical sides by $1/2$.  
This encloses area $z=(x+s)(x+u)-\frac 12 xy$, which is in $\epsilon+\Z$, so
the endpoint $(x,y,z)$ of the new path is in $H(\Z)$.  This area is smaller  than $r=(x+t)^2-\frac 12 xy$ by $1/4$, so it must be the nearest 
lattice point.  This new path may not be a geodesic, but it establishes the inequality $|(x,y,z)|\le n$.
On the other hand, shortening all four sides of $\gamma$ by $1/4$ produces a $4$-sided \CC geodesic to $(x,y,r')$ of length $n-1$,
and one easily checks that $r'<z<r$.  This shows that
$\dCC((x,y,z),\zero)>n-1$, and it follows that $|(x,y,z)|=n$.  

In these cases, we have established 
that the highest lattice point over $(x,y)$ inside the closed \CC ball of radius $n$ has word-length
exactly $n$.  Now we want to show that the lowest lattice point over $(x,y)$ outside the closed \CC ball of radius $n-2$ also has
word-length exactly $n$.  But this point has \CC norm larger than $n-2$, so it has word-length larger than $n-2$, and by parity considerations
the word-length must be exactly $n$.

The last case is easy:  for any unstable point 
$(x,y,z)\in H(\R)$, its \CC norm is $\dCC((x,y,z),\zero)=x+y$.  We claim that these points
satisfy $|(x,y,z)|_\std=x+y$ as well.
To see this, just note that starting with the two-sided path $\E_1^x\E_2^y$, which encloses
area $\frac 12 xy$ with length $x+y$, one can shave off area one unit at a time with generator swaps 
(replacing some occurrence of $\E_1\E_2$ with $\E_2\E_1$)
until the area enclosed is $z$.
This produces a word path whose length matches the \CC geodesic, so it too is geodesic.
\end{proof}

Thus far, we have shown that if $\sigma$ is a subset of $\Sreg$,
then
$|S_n\cap \Delta\sigma|\sim |H(\Z) \cap \Delta_n\sigma|$, because $S_n$ has twice as many points 
as $\Delta_n\sigma$, but on half of the lines (see Figure~\ref{cross-sec}). 
On the other hand, 
$|S_n\cap \Delta\sigma|=|H(\Z) \cap \Delta_n\sigma|$ exactly for $\sigma\subset\Suns$.
Putting these together and then using the Counting Theorem (Thm~\ref{countingthm}) 
to estimate the number of lattice points 
in $\Delta_n\sigma$, we obtain Theorem~\ref{std}.


\begin{thebibliography}{99}

\bibitem{blachere} S.\ Blach\`ere, Word distance on the discrete Heisenberg group.
Colloquium Mathematicum 95, vol. 1, 2003, 21--36.

\bibitem{breuillard} E.\ Breuillard, Geometry of groups of polynomial growth and shape of large balls.\\
{\sf arXiv:0704.0095}

\bibitem{burago} D.Yu.\ Burago, Periodic metrics.  
Representation theory and dynamical systems,  205--210, 
{\em Adv.\ Soviet Math.}, 9, Amer. Math. Soc., Providence, RI, 1992.

\bibitem{busemann} H.\ Busemann, The isoperimetric problem in the {M}inkowski plane.
AJM 69 (1947), 863--871.

\bibitem{capogna} L.\ Capogna, D.\ Danielli, S.\ Pauls and J.\ Tyson,
{\em An Introduction to the Heisenberg Group and to the Sub-Riemannian Isoperimetric Problem}. 
Birkhauser, Progress in Mathematics, 2007.

\bibitem{dani} P.\ Dani, 
The asymptotic density of finite-order elements
in virtually nilpotent groups.
Journal of Algebra 316 (2007) 54--78.

\bibitem{DLM} M.\ Duchin, S.\ Leli\`evre, and  C.\ Mooney,
The shape of spheres in free abelian groups, preprint.

\bibitem{krat} S.A.\ Krat, Asymptotic properties of the Heisenberg group.
Journal of Mathematical Sciences, Vol. 110, No. 4 (2002) 2824--2840.

\bibitem{pansu} P.\ Pansu, Croissance des boules et des g\'eod\'esiques ferm\'ees dans les nilvari\'et\'es.
 Ergodic Theory Dynam. Systems  3  (1983),  no. 3, 415--445. 

\bibitem{shapiro} M.\ Shapiro, A geometric approach to the almost convexity and growth of some nilpotent groups.  Math. Ann. 285, 601--624 (1989).

\bibitem{stoll1} M.\ Stoll,
Rational and transcendental growth series for the higher Heisenberg groups.
Invent. math. 126, 85Ð109 (1996).

\bibitem{stoll2} M.\ Stoll,
On the asymptotics of the growth of 2-step nilpotent groups.
J. London Math. Soc. (2) 58 (1998) 38--48.

\end{thebibliography}
\end{document}